\documentclass[12pt]{article}

\usepackage{amsmath, amsthm, amssymb, amscd}
\usepackage{bm, mathrsfs, dsfont}
\usepackage{cases}
\usepackage{bbding}
\usepackage{cite, hyperref}
\usepackage[all,pdf]{xy}
\usepackage{upgreek}
\usepackage{graphicx, caption, float, tikz}
\usetikzlibrary{calc, positioning, shapes.misc, arrows.meta}

\newtheorem{defn}{Definition}[section]
\newtheorem{qu}{Question}[section]

\newtheorem{lem}{Lemma}[section]
\newtheorem{prop}{Proposition}[section]
\newtheorem{thm}{Theorem}[section]
\newtheorem{cor}{Corollary}[section]
\newtheorem{rem}{Remark}[section]

\title{\bfseries Irreducible cone spherical metrics and stable extensions of two line bundles}
\author{Lingguang Li, Jijian Song and Bin Xu}

\begin{document}
\maketitle

\noindent{\small {\bfseries Abstract:}  A cone spherical metric is called {\it irreducible} if any developing map of the metric does {\it not} have monodromy in ${\rm U(1)}$. By using the theory of indigenous bundles, we construct on a compact Riemann surface $X$ of genus $g_X \geq 1$ a canonical surjective map from the moduli space of stable extensions of two line bundles to that of irreducible metrics with cone angles in $2 \pi \mathbb{Z}_{>1}$, which is generically injective in the algebro-geometric sense as $g_X \geq 2$.  As an application, we prove the following two results about irreducible metrics: \\
$\bullet$ as $g_X \geq 2$ and $d$ is even and greater than $12g_X - 7$, the effective divisors of degree $d$ which could be represented by irreducible metrics form an arcwise connected Borel subset of Hausdorff dimension $\geq 2(d+3-3g_X)$ in ${\rm Sym}^d(X)$;\\
$\bullet$ as $g_X \geq 1$, for almost every effective divisor $D$ of degree {\it odd} and greater than $2g_X-2$ on $X$, there exist finitely many cone spherical metrics representing $D$.}\\



\section{Introduction}
\label{Intro}
Lots of research works have been done on the existence and uniqueness of cone spherical metrics on compact Riemann surfaces from many aspects of mathematics, including complex analysis, PDE, synthetic geometry, etc. However, to the best knowledge of the authors, it seems that little appears to be known on the algebro-geometric side. We mainly focus on irreducible cone spherical metrics with cone angles in $2\pi{\Bbb Z}_{>1}$. Such a metric on the underlying Riemann surface gives an effective divisor $D$ supported at its cone singularities. If $p_{j}$ is a cone singularity of the metric, then the coefficient $ \beta_{j} \in {\Bbb Z}_{>0}$ of $D$ and the cone angle $\alpha_{j} \in 2\pi{\Bbb Z}_{>1}$ at $p_{j}$ are related by the equality $\beta_{j} = \frac{\alpha_{j}}{2\pi}-1$. We also say that {\it this irreducible metric represents the divisor $D$}.
We say an extension $0\to L_1\to E\to L_2\to 0$ of a line bundle $L_2$ by another line bundle $L_1$ on a compact Riemann surface ${\it stable}$ if and only if $E$ is a rank two stable vector bundle.

We establish a correspondence between irreducible metrics representing effective divisors and stable extensions of two line bundles (modulo tensoring a line bundle) on a compact Riemann surface $X$ of genus $g_X \geq 1$. Precisely speaking, we find a natural surjective map from the moduli space of stable extensions to that of irreducible metrics representing effective divisors, which  has at most $2^{2g_X}$ preimages for a fixed irreducible metric and is generically one-to-one as $g_X \geq 2$ (Theorem \ref{thm:corr}). There always exists a stable extension $0\to L_1\to E\to L_2\to 0$  of $L_2$ by $L_1$ if and only if the degree inequality
\[ \deg \, L_1 < \deg \, L_2 \]
holds, and with a slight modification on elliptic curves (Theorem \ref{thm:stable}). Moreover, we could specify the effective divisor represented by the irreducible metric corresponding to a given stable extension  $0\to L_1\to E\to L_2\to 0$ in terms of the Hermitian-Einstein metric on $E$ (Theorem \ref{thm:rami}). Therefore, we obtain a real analytic map ${\frak R}_{(L_1,L_2)}$, called the {\it ramification divisor map} (Theorem \ref{thm:expressionofR}). Furthermore, ${\frak R}_{(L_1,L_2)}$ maps the space of stable extensions of $L_2$ by $L_1$ to the complete linear system $\left|L_1^{-1}\otimes L_2 \otimes K_{X}\right|$, and its image contains exactly all the effective divisors in $\left|L_1^{-1}\otimes L_2 \otimes K_{X}\right|$ which could be represented by some irreducible metric (Corollary \ref{cor:image}). Hence, the existence and uniqueness problem of irreducible metrics representing effective divisors is boiled down to understanding the corresponding properties of the ramification divisor map. In particular, when $( \deg\,L_2-\deg\, L_1)$ is odd and positive, it is proved by the PDE method that ${\frak R}_{(L_1,L_2)}$ is a surjective map \cite{BdMM11, Troyanov91}, i.e. each effective divisor $D$ with degree odd and greater than $2g_X-2$ could be represented by at least one irreducible metric. However, the method of PDE in \cite{BdMM11, Troyanov91}  turns out to be invalid for  cone spherical metrics representing effective divisors of degree even due to the bubbling phenomena. In this manuscript, with the help of the preceding algebro-geometric framework, we could understand irreducible metrics representing effective divisors with degree being either odd or even in a unified way. Among others, we obtain \\

\noindent {\bf Theorem} (Corollaries \ref{cor:irr} and \ref{cor:finitemetric})
{\it Let $X$ be a compact Riemann surface of genus $g_X>1$ and $D$ an effective divisor of degree $d>2g_X-2$.
\begin{enumerate}

\item If $d$ is even, then there exists an irreducible metric representing some effective divisor linearly equivalent to $D$.

\item If $d>12g_X-7$ is even, all the effective divisors as above form an arcwise connected Borel subset of Hausdorff dimension $\geq 2(d+3-4g_X)$ in $|D|$, and all the effective divisors of degree $d$ represented by irreducible metrics form an arcwise connected Borel subset of Hausdorff dimension $\geq 2(d+3-3g_X)$ in ${\rm Sym}^d(X)$.

\item If $d$ is odd, then, for almost every effective divisor $D$ on $X$, there exist finitely many cone spherical metrics representing $D$ on $X$.
\end{enumerate}}

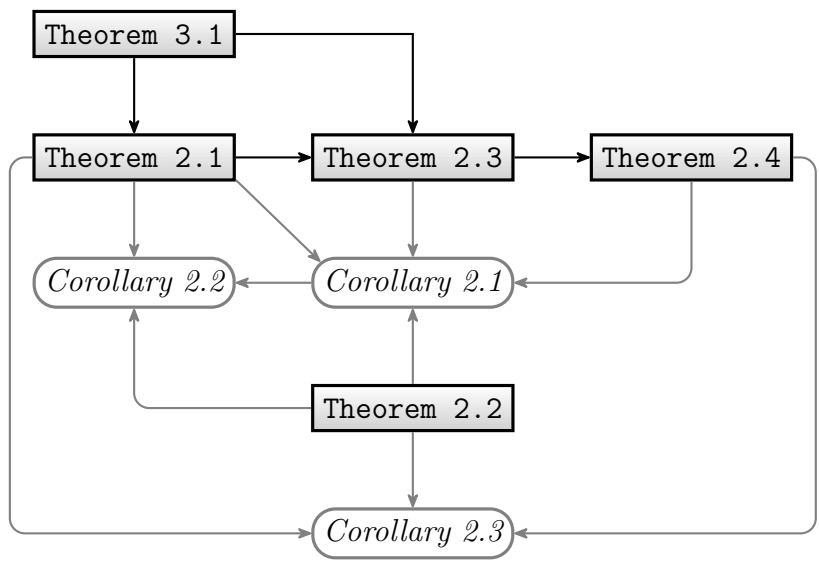
\begin{figure}[H]
\centering
\begin{tikzpicture}[
>={Stealth[round]},thick,
theorem/.style={rectangle, minimum size=6mm, very thick, draw=black, top color=white, bottom color=black!20, font=\ttfamily},
corollary/.style={rectangle, minimum size=6mm, rounded corners=3mm, very thick, draw=gray, top color=white, bottom color=white, font=\itshape}]

\node (thm3-1) [theorem] {Theorem \ref{thm:correspondence}};
\node (thm2-1) [theorem, below=of thm3-1] {Theorem \ref{thm:corr}};
\node (thm2-3) [theorem, right=of thm2-1] {Theorem \ref{thm:rami}};
\node (thm2-4) [theorem, right=of thm2-3] {Theorem \ref{thm:expressionofR}};
\node (cor2-1) [corollary, below=of thm2-3] {Corollary \ref{cor:image}};
\node (thm2-2) [theorem, below=of cor2-1] {Theorem \ref{thm:stable}};
\node (cor2-2) [corollary, below=of thm2-1] {Corollary \ref{cor:irr}};
\node (cor2-3) [corollary, below=of thm2-2] {Corollary \ref{cor:finitemetric}};

\path (thm3-1) edge[->] (thm2-1)
         (thm2-1) edge[->] (thm2-3)
         (thm2-3) edge[->] (thm2-4);
\draw[->] (thm3-1.east) -| (thm2-3);

\path (thm2-1) edge[gray, ->] (cor2-2)
         (thm2-3) edge[gray, ->] (cor2-1)
         (cor2-1) edge[gray, ->] (cor2-2)
         (thm2-2) edge[gray, ->] (cor2-1)
         (thm2-2) edge[gray, ->] (cor2-3)
         ($(thm2-1.east) - (0, 3mm)$) edge[gray, ->] ($(cor2-1.west) + (1.2mm, 2.8mm)$);
\draw[rounded corners=2mm, gray, ->] (thm2-2.west) -| (cor2-2);
\draw[rounded corners=2mm, gray, ->] (thm2-4.south) |- (cor2-1);
\draw[rounded corners=2mm, gray, ->] (thm2-4.east) -- ($(thm2-4.east) + (3mm,0)$) |- (cor2-3);
\draw[rounded corners=2mm, gray, ->] (thm2-1.west) -- ($(thm2-1.west) - (3mm, 0)$) |- (cor2-3);

\end{tikzpicture}
\caption{Leitfaden}
\label{fig:Leitfaden}
\end{figure}

We conclude this section by explaining the organization of the left five sections of this manuscript.
In the next section, we give the relevant background of cone spherical metrics in detail and state precisely the main results consisting of Theorems 2.1-4 and Corollaries 2.1-3 in a logical order. In particular, we prove Corollaries 2.2-3 directly there.
In Section \ref{sec:bundles}, we establish a correspondence in Theorem \ref{thm:correspondence}
between cone spherical metrics representing effective divisors and projective bundles associated to rank two polystable vector bundles with sections which are not locally flat, which contains the correspondence in Theorem \ref{thm:corr} as a special case. Among others, by using algebra of vector bundles, we prove the whole of Theorem \ref{thm:corr} in this section. In Section \ref{sec:langetype}, we describe an embedding of $X$ into a projective space induced by $H^0(X, K_X\otimes L_{2} \otimes L_{1}^{-1})$, which we use to give an elementary proof of Theorem \ref{thm:stable}. In Section \ref{sec:ramidivmap}, by using differential geometry on vector bundles, we prove Theorem \ref{thm:rami}, Theorem \ref{thm:expressionofR} and Corollary \ref{cor:image}. We discuss the relation of this manuscript to existing works and some open questions in the last section. Figure \ref{fig:Leitfaden} expresses the internal logical relationship lying in the main results of this manuscript.

\section{Background and main results}
\label{main}
\subsection{Cone spherical metrics}
\label{sec:knownresult}
Let $X$ be a compact Riemann surface of genus $g_X$ and $D = \sum_{j=1}^n \, \beta_j \, p_j$ an $\mathbb{R}$-divisor on $X$ such that $p_1, \cdots, p_n$ are $n \geq 1$ distinct  points on $X$ and $0 \neq \beta_j > -1$. We call a smooth conformal metric $g$ on $X \setminus {\rm supp} \, D := X \setminus \{p_1, \cdots, p_n\}$ a {\it conformal metric representing $D$ on $X$} if, for each point $p_j$, there exists a complex coordinate chart $(U_{j}, \, z_{j})$ centered at $p_j$ such that the restriction of $g$ to $U_{j} \setminus \{ p_j \}$ has form $e^{2 \varphi_{j} } \, |dz_{j}|^2$, where the real valued function $\varphi_{j} - \beta_j \, \ln \, |z_{j} - z_{j}(p_{j})|$ extends to a continuous function on $U_{j}$. In other words,  $g$ {\it has a conical singularity at each $p_j$ with cone angle $2 \pi( 1 + \beta_j )$}.  The generalized Gauss-Bonnet formula
\begin{equation}
\label{equ:GB}
 \frac{1}{2\pi}\int_{X\setminus {\rm supp}\, D}\, K_g \, {\rm d}A_g = 2-2g_X+\deg\, D
\end{equation}
holds \cite[Proposition 1]{Troyanov91} and \cite[Theorem 1]{Fang2019} as the Gaussian curvature function $K_g$ of $g$ is integrable on $X \setminus {\rm supp}\, D$ and
we denote the degree of $D$ by $\deg \, D := \sum_{j=1}^n \, \beta_j$. We call $g$ a {\it cone spherical metric representing $D$} if $K_g \equiv +1$ outside ${\rm supp} \, D=\{p_1, \cdots, p_n\}$. We also note that the PDEs satisfied by cone spherical metrics form a special class of mean field equations, which are relevant to both Onsager's vortex model in statistical physics \cite{CLMP92} and the Chern-Simons-Higgs equation in superconductivity \cite{CLW2004}. Similarly, we could define {\it cone hyperbolic metrics} or {\it cone flat ones} if their Gaussian curvatures equal identically $-1$ or $0$ outside the conical singularities. People naturally came up with

\begin{qu}
\label{qu:hfs}
Characterize all real divisors with coefficients lying in $(-1,\, \infty)\setminus\{0\}$ on $X$ which could be represented by cone hyperbolic, flat or spherical metrics, respectively.
\end{qu}

The generalized Gauss-Bonnet formula \eqref{equ:GB} gives for Question \ref{qu:hfs} a natural necessary condition of
\begin{equation}
\label{equ:nece}
 {\rm sgn} \big (2 - 2g_X + \deg \, D \big ) = {\rm sgn}(K_g).
\end{equation}
It is also sufficient for the cases of hyperbolic and flat metrics, and the hyperbolic or flat metric representing $D$ on $X$ exists uniquely up to scaling  \cite{Hei62, McOwen88, Tro86, Troyanov91}. The history of the research works of cone hyperbolic metrics goes back to {\' E}. Picard \cite{Pi1905} and H. Poincar{\' e} \cite{Po1898}. However,  \eqref{equ:nece} is {\it not} sufficient for the existence of  cone spherical metrics \cite{Tro89}. And the situation is even worse in the spherical case that the uniqueness result does not hold in general \cite[Theorem 1.5]{CWWX2015}. As a consequence, the spherical case of Question \ref{qu:hfs} is still widely open although many mathematicians had attacked or have been investigating it by using various methods and obtained a good understanding of the question \cite{BdMM11, CLW14, CWWX2015, Er04, Er1706, EGT1405, EG15, EGT1409, EGT1504, Er1905, Luo93, MP1505, MP1807, MZ1710, MZ1906, SCLX2017, Troyanov91, UY2000, XuZhu19, Zhu1902}.

In this subsection and the next, we list some of the known results which are relevant to this manuscript. Troyanov proved an existence theorem \cite[Theorem C]{Troyanov91} on the problem of prescribing the Gaussian curvature on surfaces with the conical singularities given in subcritical regimes. Troyanov's theorem implies that there exists a cone spherical metric representing the $\mathbb{R}$-divisor $D=\sum_{j=1}^n\,\beta_j p_j$ with $-1<\beta_j\not=0$ on $X$ if
\[ 0 < 2 - 2g_X + \deg \, D < \min \, \Big\{2, \, 2 + 2 \min_{1 \leq  j \leq n} \, \beta_j \Big\}. \]
Bartolucci-De Marchis-Malchiodi obtained a very general existence theorem \cite[Theorem 1.1]{BdMM11} on the same problem in supercritical regimes. In particular, they showed that there exists a cone spherical metric representing the effective divisor $D = \sum_{j = 1}^n \, \beta_j p_j$ on $X$ if the following conditions hold:
\begin{itemize}
\item $g_X \geq 1$ and $\beta_j>0$ for all $1\leq j\leq n$,
\item $2-2g_X+\deg\, D$ is greater than $2$ and does not belong to the  discrete subset
$\left \{ \mu > 0 \mid \mu = 2k + 2 \sum_{j=1}^n n_j (1 + \beta_j ), k \in \mathbb{Z}_{\geq 0}, n_j \in \{ 0,1\} \right \}$ of ${\Bbb R}$.
\end{itemize}
Combining the results by Troyanov and Bartolucci-De Marchis-Malchiodi, we could obtain that if $D = \sum_{j=1}^n\,\beta_j p_j$ is an effective divisor of  degree {\it odd} on a compact Riemann surface $X$ of genus $g_X \geq 1$, then there always exists a cone spherical metric representing $D$ provided the natural necessary condition of $\deg \, D>2g_X-2$ holds. This existence result on elliptic curves for cone spherical metrics was also obtained independently by Chen-Lin  as a corollary of a more general existence theorem \cite[Theorem 1.3]{CL15} for a class of mean field equations of Liouville type with singular data.

\subsection{Irreducible  metrics}
In this manuscript we would like to focus on cone spherical metrics with cone angles in $2\pi \mathbb{Z}_{>1}$, i.e. cone spherical metrics representing effective divisors. We shall establish an algebro-geometric framework for such metrics and obtain some new existence results of them as an application. In order to state them in detail, we need to prepare some notations.

We give a quick review of developing maps of cone spherical metrics representing effective divisors and recall the concept of reducible/irreducible (cone spherical) metrics \cite{CWWX2015, Er04, UY2000}. We call a non-constant multi-valued meromorphic function $f \colon X \to \mathbb{P}^1=\mathbb{C}\cup \{\infty\}$ {\it a projective function on $X$} if the monodromy of $f$ lies in the group ${\rm PSL}(2,\,\mathbb{C})$ consisting of all M{\" o}bius transformations. Then, for a projective function $f$ on $X$, we could define its {\it ramification divisor} $R(f)$, which is an effective divisor on $X$. It was proved in \cite[Section 3]{CWWX2015} that there is a cone spherical metric representing an effective divisor $D$ on $X$ if and only if there exists a projective function $f$ on $X$ such that $R(f)=D$ and the monodromy of $f$ lies in
\[ {\rm PSU(2)}:=\left\{ z \mapsto \frac{az+b}{-\overline{b}z + \overline{a}} \; \Big| \; |a|^2 + |b|^2 = 1 \right\} \subset {\rm PSL}(2,\,\mathbb{C}) \]
(we call that {\it $f$ has projective unitary monodromy} for short later on), and $g$ equals the pull-back $f^*g_{\rm st}$ of the standard conformal metric $g_{\rm st}=\frac{4|dw|^2}{(1+|w|^2)^2}$ on $\mathbb{P}^{1}$ by $f$. At this moment, we call $f$ a {\it developing map} of the metric $g$, which is unique up to a pre-composition with a M{\" o}bius transformation in ${\rm PSU}(2)$. In particular,  it is well known that effective divisors represented by cone spherical metrics on the Riemann sphere $\mathbb{P}^1$ are exactly ramification divisors of rational functions on $\mathbb{P}^1$ \cite[Theorem 1.9]{CWWX2015}, and hence all of them have even degree. Recalling the universal (double) covering $\pi \colon {\rm SU(2)}\to {\rm PSU(2)}$, we make an observation (Corollary \ref{cor:criterion}) that an effective divisor $D$ represented by a cone spherical metric $g$ on $X$ has {\it even} degree if and only if the monodromy representation $\rho_f \colon \, \pi_1(X) \to {\rm PSU(2)}$ of a developing map $f$ of the metric $g$ could be lifted to a group homomorphism $\widetilde{\rho_f} \colon \, \pi_1(X) \to {\rm SU}(2)$ such that there holds the following commutative diagram
\[ \xymatrix{
                                                                                   & {\rm SU(2)} \ar[d] ^{\pi} \\
\pi_{1}(X) \ar[ru]^{\widetilde{\rho_f}} \ar[r]_{\rho_{f}} & {\rm PSU(2)}.
}\]

A cone spherical metric  is called {\it reducible} if and only if some developing map of the metric has monodromy in ${\rm U}(1) = \left\{ z \mapsto e^{\sqrt{-1}\, t}z \mid t \in [0, \, 2\pi) \right\}$. Otherwise, it is called {\it irreducible}. Q. Chen, W. Wang, Y. Wu and B.X. \cite[Theorems 1.4-5]{CWWX2015} established a correspondence between meromorphic one-forms with simple poles and periods  in $\sqrt{-1}\mathbb{R}$ and general reducible cone spherical metrics, whose cone angles do not necessarily lie in $2\pi \mathbb{Z}_{>1}$. In particular, an effective divisor $D$ represented by a reducible metric must have {\it even} degree since reducible metrics satisfy the lifting property in the last paragraph \cite[Lemma 4.1]{CWWX2015}. For simplicity, we may look at the degree of an effective divisor represented by a cone spherical metric {\it the degree of the metric}. Recall the fact in Subsection \ref{sec:knownresult} that if $D$ is an effective divisor with degree being odd and greater than $2g_X-2$ on a compact Riemann surface $X$ of genus $g_X \geq 1$, there always exists a cone spherical metric representing $D$ on $X$. However,  the PDE method used in its proof seems {\it invalid} for the case of even degree. The reason lies in that there exists no a priori $C^0$ estimate for the corresponding PDE due to the blow-up phenomena caused by the one-parameter family of reducible metrics representing the same effective divisor of even degree \cite[Theorems 1.4-5]{CWWX2015}.

\subsection{An algebro-geometric framework for irreducible metrics}

Motivated by the theory of indigenous bundles initiated by R. C. Gunning \cite{Gunning67} and developed by  R. Mandelbaum \cite{Mand72,Mand73}, we take a new approach to understand cone spherical metrics representing effective divisors of degree either odd or even in a unified way.
Not only does the following algebro-geometric framework for irreducible metrics representing effective divisors shed new light on the connection between differential geometry and algebraic geometry underlying these metrics, but also it plays a crucial role for the existence problem of irreducible metrics of even degree. We postpone the explanation of the concepts of various relevant holomorphic bundles to Section \ref{sec:bundles}.

To state the theorem on this framework, we need prepare some notations. Let $X$ be a compact Riemann surface of genus $g_X\geq 1$. We denote by ${\mathcal {SE}}(X)$ the moduli space of stable extensions $0\to L_1\to E\to L_2\to 0$ of two arbitary line bundles $L_1$ and $L_2$ over $X$ modulo the process of tensoring line bundles. Observing that ${\mathcal {SE}}(X)$ has a natural stratification structure with respect to the positive index $k:=\deg\, E-2\deg\,L_1$, we denote the corresponding stratum by ${\mathcal {SE}}_k(X)$. We denote
by ${\mathcal {MI}}(X,\,{\Bbb Z})$ the space of irreducible metrics representing effective divisors on $X$.

\begin{thm}
\label{thm:corr}
Let $X$ be a compact connected Riemann surface of genus $g_X \geq 1$. Then there exists a canonical surjective map $\sigma \colon {\mathcal {SE}}(X)\to {\mathcal {MI}}(X,\,{\Bbb Z})$ such that
\begin{enumerate}

\item the irreducible metric corresponding to an equivalence class of a stable extension $0\to L_1\to E\to L_2\to 0$ represents an effective divisor lying  in the complete linear system $\left| L_1^{-2} \otimes \det E \otimes K_{X} \right|$, where two stable extensions are called {\rm equivalent} if and only if one of them is obtained from the other by tensoring a line bundle;

\item an irreducible metric representing an effective divisor has at most $2^{2g_{X}}$ preimages under $\sigma$; and

\item if $g_X \geq 2$ and the index $k >  10g_{X} - 5$, the restriction of $\sigma$ to some Zariski open subset of ${\mathcal {SE}}_k(X)$ is injective.
\end{enumerate}
\end{thm}

\begin{rem}
\label{rem:Gro}
{\rm The proof of Theorem \ref{thm:corr} is given in the last latter part of Subsection \ref{subsec:pf}. There exists another correspondence between reducible metrics representing effective divisors and certain unstable and polystable extensions of two line bundles. Actually in Section \ref{sec:bundles} we encapsulate it and the surjective property of $\sigma \colon {\mathcal {SE}}(X)\to {\mathcal {MI}}(X,\,{\Bbb Z})$  in  Theorem \ref{thm:corr} into Theorem \ref{thm:correspondence}, and explain the exact meaning of the correspondence in the proof of Theorem \ref{thm:correspondence}, which occupies the former part of  Subsection \ref{subsec:pf}.
}\end{rem}

\subsection{Existence of stable extensions}
We see clearly that a stable extension $0\to L_1\to E\to L_2\to 0$ gives an embedding $L_1\to E$. On the other hand, we observe that if $L \to E$ is an embedding of a line bundle $L$ into a rank two stable bundle $E$ on $X$, called a {\it stable embedding},  then $L \to E$ gives a stable extension of $L^{-1} \otimes \det E$ by $L$, i.e.
\[ 0 \to L\rightarrow E \rightarrow L^{-1} \otimes \det E \to 0 \quad \text{with $E$ stable}. \]
Therefore, we may use both the embedding $L \to E$ and the preceding stable extension interchangeably in what follows.
Atiyah \cite{Atiyah57II} proved the fact that there exists no rank two stable vector bundle of even degree on  elliptic curves. Combining this fact and Theorem \ref{thm:corr}, we observe that {\it cone spherical metrics of even degree are all reducible on elliptic curves}. In order to find irreducible metrics of even degree,  we are naturally motivated to prove on a compact Riemann surface of positive genus  the following theorem, which is relevant to and more refined than the Lange conjecture \cite[p.452]{Lan83}, solved by Russo-Teixidor i Bigas \cite{RT99} and Ballico-Russo \cite{Ba2000, BR98}.
We also admit that this Lange-type theorem should be well known to experts. Actually we found a special case of it in \cite[Corollary 1.2]{LN83}.

\begin{thm}
\label{thm:stable}
Let $X$ be a compact Riemann surface of genus $g_{X}\geq 1$, and let $L_{1}, L_{2}$ be two line bundles of degrees $d_{1}, d_{2}$ respectively. Suppose that either one of the following conditions holds:
\begin{itemize}
\item $d_{2} - d_{1}$ is a positive odd integer and $g_{X}=1$;
\item $d_{2} - d_{1}$ is a positive integer and $g_{X} \geq 2$.
\end{itemize}
Then there always exists a stable extension $E$ of $L_{2}$ by $L_{1}$.
Moreover, the set of unstable extensions forms an affine subvariety of codimension $\geq g_{X}$ {\rm \big(}resp. $\geq (g_{X} - 1)${\rm \big)} in the extension space $\operatorname{Ext}_{X}^{1}(L_{2}, L_{1})$ if $d_{2} - d_{1}$ is odd (resp. even).
Furthermore, if $d_{2} - d_{1} = 1$, then each nontrivial extension of $L_2$ by $L_1$ is stable.
\end{thm}

The proof of this Lange-type theorem is given in Section \ref{sec:langetype}.

\subsection{Effective divisors represented by irreducible metrics}
\label{subsec:rami}

In this subsection we specify the effective divisor represented by the irreducible metric corresponding to a stable embedding $L \to E$ in Theorem \ref{thm:corr} by introducing the so-called {\it ramification divisor map}, where we use the Hermitian-Einstein metric on $E$. As a consequence, we could characterize
all the effective divisors represented by irreducible metrics in a conceptual sense.

Let $E$ be a stable extension of $M := L^{-1} \otimes \det E$ by $L$. Since $E$ is stable,  we know that $\mathbb P(E)$ arises from an irreducible projective unitary representation of $\pi_{1}(X)$ \cite{NS1965}. Hence, for a given smooth K\" ahler form $\omega_X$ on $X$, there exists on $E$ a unique Hermitian-Einstein metric $h$ up to scaling such that its Chern connection $D_E$ satisfies
\[ D_{E} \circ D_{E} = \lambda \, {\rm Id}_E,\quad {\rm where}\quad \lambda=-\sqrt{-1} \,\big(\deg E\big) \,\omega_X. \]
Let us rewrite $D_E$ as
$D_{E} = \partial_{E} + \bar{\partial}_{E}$,
where $\bar{\partial}_{E}$ is the complex structure of $E$, $\mathcal{A}^{p,q}(E)$ is the sheaf of smooth $E$-valued $(p,q)$-forms and $\partial_{E}\colon \mathcal{A}^{0}(E) \rightarrow \mathcal{A}^{1,0}(E)$ is the $(1,0)$-part of $D_{E}$.
Looking at the stable extension $0\to L\to E\to M \to 0$ as an element in
 ${\rm Ext}_{X}^1(M,\,L)=H^1\big(X,\, L\otimes M^{-1}\big)$ \cite[Proposition 2]{Atiyah57I}, we denote it by
${\Bbb E}$.
Then we obtain the following commutative diagram
\[\xymatrix{
\mathcal A^{0}(L)\ar[r] ^-{i}\ar[drr]_-{\theta_{\Bbb E}} & \mathcal A^{0}(E)\ar[r] ^-{\partial_{E}} & E\otimes \mathcal{A}^{1,0}\ar[d]^-{p \otimes \operatorname{id}} \\
 & & M \otimes \mathcal{A}^{1,0}
}\]
where $p \otimes \operatorname{id}$ is induced by $E \rightarrow E/L \cong M$ and $\theta_{\Bbb E} := (p \otimes \operatorname{id}) \circ \partial_{E} \circ i$. In the former part of Section \ref{sec:ramidivmap}, we  prove the following theorem which establishes a connection between $\theta_{\Bbb E}$ and the effective divisor represented by the cone spherical metric defined by $L\to E$.
\begin{thm}
\label{thm:rami}
Suppose that $L, \, M$ are two line bundles on a compact Riemann surface $X$ of genus $g_{X} \geq 1$ and
${\Bbb E}\in H^1\big(X,\, L\otimes M^{-1}\big)$ is a stable extension of $M$ by $L$. Then we have
\begin{enumerate}
\item The map $\theta_{\Bbb E}$ is $\mathcal{O}_{X}$-linear, i.e. $\theta_{\Bbb E}$ lies in $\operatorname{Hom}_{\mathcal{O}_{X}} \big(L, M \otimes K_{X}\big) $.
\item Under the correspondence of Theorem \ref{thm:corr}, the effective divisor $D$ represented by the irreducible metric defined by ${\Bbb E}$ coincides with the divisor $\operatorname{Div} (\theta_{\Bbb E})$ associated to the holomorphic section
    $\theta_{\Bbb E}\in \operatorname{Hom}_{\mathcal{O}_{X}} \big(L, M \otimes K_{X}\big)\cong H^{0}\big(X, L^{-1} \otimes M \otimes K_{X}\big)$.
\end{enumerate}
\end{thm}

We denote by $H^{1}(X, L \otimes M^{-1})^{s}$ the set of all stable extensions in  $H^{1}(X, L \otimes M^{-1}) \cong \operatorname{Ext}_{X}^{1}(M, L)$, which forms a
Zariski open subset of $H^{1}(X, L \otimes M^{-1})$ by Theorem \ref{thm:stable}.
Since the map ${\Bbb E}\mapsto \theta_{\Bbb E}$ given in Theorem \ref{thm:rami} satisfies that
\[\mu{\Bbb E}\mapsto \frac{\theta_{\Bbb E}}{\mu}\quad{\rm where}\quad \mu\in {\Bbb C} \setminus \{0\},\]
we call
\begin{align*}
{\frak R}_{(L, M)}\colon \mathbb{P}\Big(H^{1}\big(X, L \otimes M^{-1}\big)^{s}\Big) &\rightarrow \mathbb{P}\Big(H^{0}\big(X, L^{-1} \otimes M \otimes K_{X}\big)\Big), \\
[{\Bbb E}] &\mapsto {\rm Div}(\theta_{\Bbb E})
\end{align*}
the {\it ramification divisor map} associated to the two line bundles $L$ and $M$.
Therefore, we could reduce the existence problem, the uniqueness problem and the counting problem of irreducible metrics representing an effective divisor $D$ to understanding the corresponding properties of the ramification divisor map associated to the two line bundles $L$ and $M$ such that ${\mathcal O}_{X}(D)\cong L^{-1}\otimes M\otimes K_X$.
We describe such maps in terms of differential geometry of vector bundles.
Let us consider a stable extension
\begin{equation}
\label{equ:extension}
{\Bbb E}:\quad 0 \to L \rightarrow E \rightarrow M \to 0
\end{equation}
in $H^{1}(X, L \otimes M^{-1})$. Let $h$ be the preceding Hermitian-Einstein metric on $E$. We denote by $L^{\perp}$ the orthogonal complement of $L$ in $E$. Then the quotient line bundle $M$ is isomorphic to $L^{\perp}$ as complex vector bundles, however they are not necessarily isomorphic as holomorphic vector bundles. Both $L$ and $M$ inherit hermitian structures from $E$ in the usual way. We denote by $D_{L}$ and $D_{M}$  the Chern connections of $L$ and $M$, respectively. The Chern connection $D_{E}$ of $E$ could be expressed by
\[ D_{E} = \begin{pmatrix}
D_{L} & -\beta \\
\beta^{*_{h}} & D_{M}
\end{pmatrix}, \]
where $\beta^{*_{h}}$ is the second fundamental form of $L$ in $E$. Then we have

\begin{thm}
\label{thm:expressionofR} Using the notations in the preceding paragraph, we have that
$\beta$ is a representative 1-cocycle of the stable extension \eqref{equ:extension} in $H^{0,1}_{\bar{\partial}}(X, L \otimes M^{-1}) \cong H^{1}(X, L \otimes M^{-1})$. By the canonical isomorphism between $H^{1,0}_{\bar{\partial}}(X, L^{-1} \otimes M)$ and $H^{0}(X, L^{-1} \otimes M \otimes K_{X})$, we could express the ramification divisor map as
\begin{equation}
\label{equ:rami}
{\frak R}_{(L, M)} \left( [\beta] \right) = [\beta^{*_{h}}]=:[\beta^{*_{\beta}}],
\end{equation}
where we define $[\beta^{*_{\beta}}]$ as $[\beta^{*_{h}}]$ since $h$ is determined by $\beta$.
Moreover, ${\frak R}_{(L, M)}$ is a real analytic map.

\end{thm}

The proof of Theorem \ref{thm:expressionofR} is given in the middle part of Section \ref{sec:ramidivmap}.

\begin{cor}
\label{cor:image}
Under the assumptions of Theorem \ref{thm:rami}, a divisor $D$ in the complete linear system $\left| L^{-1} \otimes M \otimes K_{X} \right|$ can be represented by an irreducible metric if and only if $D$ lies in the image  $\operatorname{Im} \big({\frak R}_{(L, M)}\big)$ of ${\frak R}_{(L, M)}$. Moreover,
 $\operatorname{Im} \big({\frak R}_{(L, M)}\big)$ is an arcwise connected Borel subset of Hausdorff dimension $\geq 2\big(\deg\,M-\deg\, L+1-2g_X\big)$ in $\mathbb{P}\Big(H^{0}\big(X, L^{-1} \otimes M \otimes K_{X}\big)\Big)$ if $(\deg\, M-\deg\, L) > 10g_{X} - 5$.
\end{cor}

By the definition of ${\frak R}_{(L, M)}$, the first statement of this corollary holds automatically.  The last statement of it is non-trivial and will be proved in the ending of
Section 5 by using Theorem \ref{thm:expressionofR}.
As an application,  Theorems \ref{thm:corr} and \ref{thm:stable} and the above corollary give a new class of irreducible metrics of even degree.

\begin{cor}
\label{cor:irr}
Let $D$ be an effective divisor of degree ${\rm even}=d>2g_X-2$ on a compact Riemann surface $X$ of genus $g_X\geq 2$.
\begin{enumerate}

\item There exists an irreducible metric on $X$ representing some effective divisor linearly equivalent to $D$.

\item If $d> 12g_{X} - 7$, then the effective divisors as above form an arcwise connected Borel subset of Hausdorff dimension $\geq 2(d+3-4g_X)$ in $|D|$, and all the effective divisors of degree $d$ represented by irreducible metrics form an arcwise connected Borel subset of Hausdorff dimension $\geq 2(d+3-3g_X)$ in ${\rm Sym}^d(X)$.
\end{enumerate}
\end{cor}
\begin{proof} By the very condition of $D$, we could choose a negative line bundle $L$ such that $L^{-2} \otimes K_{X} = \mathcal{O}_{X}(D)$. Then, by Theorem \ref{thm:stable}, we could choose a stable extension $E$ of $L^{-1}$ by $L$. By Theorem \ref{thm:corr}, the embedding $L \to E$ defines an irreducible metric representing some divisor in the complete linear system $|D|$.  By using the ramification divisor map ${\frak R}_{(L,L^{-1})}$, we obtain the second statement from Corollary \ref{cor:image} and the natural holomorphic fibration
${\Bbb P}^{d-g_X}\to \operatorname{Sym}^{d}(X)\to {\rm Pic}^{d}(X)$
as $d > 12g_{X} - 7$.
\end{proof}


\begin{cor}
\label{cor:finitemetric}
Let $X$ be a compact Riemann surface of genus $g_X>0$ and $d > 2g_X-2$ an odd integer.
For almost every (a.e.) effective divisor $D$ in $\operatorname{Sym}^{d}(X)$,
 there exist finitely many cone spherical metrics representing $D$, where ``a.e.'' is in the sense of the Riemann-Lebesgue measure on the $d$-fold symmetric product ${\rm Sym}^{d}(X)$ of $X$.
\end{cor}
\begin{proof}
Let $g$ be a cone spherical metric representing $D$ on $X$. By Corollary \ref{cor:criterion}, $g$ is irreducible since $\deg D$ is odd. By the surjectivity in Theorem \ref{thm:corr}, we could assume that $g$ is defined by an embedding $\mathcal{O}_{X} \to E$. Moreover, we could see that
\[ \det E \otimes K_{X} = \mathcal{O}_{X}(D), \]
and $E$ is a stable extension of $M = K_{X}^{-1} \otimes \mathcal{O}_{X}(D)$ by $\mathcal{O}_{X}$.

We prefer to prove the case $d=2g_X-1$ at first since its proof is simpler than the general case.
Since $d=2g_X-1$,  $\deg M = 1$ and the domain of ${\frak R}_{(\mathcal{O}_{X}, M)}$ turns out to be the whole of $\mathbb{P}\Big(H^{1}(X, M^{-1})\Big) \cong \mathbb{P}^{(g_{X} - 1)}$ by Theorem \ref{thm:stable}. By Theorem \ref{thm:expressionofR}, we know that ${\frak R}_{(\mathcal{O}_{X}, M)} \colon \mathbb{P}^{(g_{X} - 1)} \to \mathbb{P}^{(g_{X} - 1)} \cong \mathbb{P}\big(H^{0}(X, M \otimes K_{X}) \big)$ is a real analytic map, which is surjective by using a theorem of Troyanov \cite[Theorem C]{Troyanov91}. If $D \in \mathbb{P}^{(g_{X} - 1)}$ is a regular value of ${\frak R}_{(\mathcal{O}_{X}, M)}$, then ${\frak R}_{(\mathcal{O}_{X}, M)}^{-1}(D)$ is a finite set. By Theorem \ref{thm:corr} again, there exist at most finitely many irreducible  metrics representing $D$ on $X$. We complete the proof by the natural fibration expression
\[{\Bbb P}^{(g_X-1)}\to \operatorname{Sym}^{(2g_{X} - 1)}(X)\to {\rm Pic}^{(2g_X-1)}(X)\]
of $\operatorname{Sym}^{(2g_{X} - 1)}(X)$, the Sard theorem and the Fubini theorem.

Suppose that $d > 2g_X - 1$. By Theorem \ref{thm:expressionofR},  ${\frak R}_{(\mathcal{O}_{X}, M)}$ could be thought of as a real analytic map from a Zariski open subset of ${\Bbb P}^{d-g_X}$ to ${\Bbb P}^{d-g_X}$, which is also surjective by using a theorem of Bartolucci-De Marchis-Malchiodi \cite[Theorem 1.1]{BdMM11}. It follows from the similar argument as the preceding case that for a regular value $D$ of ${\frak R}_{(\mathcal{O}_{X}, M)}$, there exist at most countably many irreducible metrics representing $D$ on $X$, which form a discrete metric space. The finiteness of such metrics follows from that
 the space of cone spherical metrics representing $D$ is compact \cite[Theorem, p.6]{BT02} and \cite[Theorem 1.16]{MP1807}.
\end{proof}

\begin{rem}
\label{rem:ell}
{\rm It is well known that on each elliptic curve, there exists a unique spherical metric with a cone angle of $4\pi$ \cite[Section 2]{CLW14}, which also follows from the similar argument in the second paragraph of the proof for Corollary \ref{cor:finitemetric}. Moreover, Chai-Lin-Wang \cite{CLWpre} gave a uniform upper bound in terms of $\deg\, D$ for the number of spherical metrics representing an effective divisor $D$ on elliptic curves, provided $\deg\,D$ is odd.}  
\end{rem}

\section{Cone spherical metrics and indigenous bundles}
\label{sec:bundles}
We observed in Section \ref{main} that cone spherical metrics representing effective divisors  are equivalent to projective functions with projective unitary monodromy on compact Riemann surfaces, which naturally give {\it branched projective coverings} and {\it indigenous bundles} (see their definitions in Subsection \ref{subsec:branch}) on the Riemann surfaces. Moreover, such indigenous bundles are the associated projective bundles of rank two polystable vector bundles by the projective unitary monodromy property. In this way, we could establish a more general correspondence in Theorem \ref{thm:correspondence} than the one in Theorem \ref{thm:corr} in the sense that the former could handle both irreducible and reducible metrics.
We state Theorem \ref{thm:correspondence} in Subsection \ref{subsec:state}, prepare the notions and a lemma for it in Subsection \ref{subsec:branch}, and prove it in Subsection \ref{subsec:pf}.

\subsection{The correspondence between cone spherical metrics and projective bundles}
\label{subsec:state}

We need to prepare the notion of projective unitary flat $\mathbb P^{r}$-bundles on Riemann surfaces before stating the above-mentioned correspondence. Let $P$ be a holomorphic $\mathbb P^{r}$-bundle on a compact Riemann surface $X$. We call that $P$ {\it has a projective unitary flat trivializations} if there exists a collection of trivializations $\psi_{\alpha}\colon P|_{U_{\alpha}} \rightarrow U_{\alpha} \times \mathbb{P}^r$ such that the corresponding transition functions
\[ g_{\alpha\beta} \colon U_{\alpha} \cap U_{\beta} \rightarrow {\rm PSU}(r+1) \subseteq  {\rm PSL}(r+1, \mathbb C) \]
are constant, where
\[ \psi_{\alpha} \circ \psi_{\beta}^{-1}(x, v) = \big(x, g_{\alpha\beta}(x) \cdot v \big) \quad \forall x \in U_{\alpha} \cap U_{\beta}, v\in \mathbb{P}^{r}. \]
Two such trivializations, after refinement of the covering if necessary, $\{U_{\alpha}, \psi_{\alpha}\}$ and $\{U_{\alpha}, \tilde{\psi}_{\alpha} \}$ are called equivalent if there exists a collection of maps $\varphi_{\alpha}\colon U_{\alpha} \times \mathbb{P}^r \rightarrow U_{\alpha} \times \mathbb{P}^r$ such that
\begin{itemize}
  \item $\varphi_{\alpha}(x, v)=\big( x, g_{\alpha}(x) \cdot v \big)$, where $g_{\alpha}\colon U_{\alpha} \rightarrow {\rm PSU}(r+1)$ is constant.
  \item For any $\alpha$ and $\beta$, $\varphi_{\alpha} \circ \psi_{\alpha} \circ \psi_{\beta} ^{-1} = \tilde{\psi}_{\alpha} \circ \tilde{\psi}_{\beta} ^{-1}\circ \varphi_{\beta}$ on $\psi_{\beta}(P|_{U_{\alpha} \cap U_{\beta}})$.
\end{itemize}
Then an equivalence class of such trivializations is called a {\it projective unitary flat structure} on $P$. A holomorphic $\mathbb P^{r}$-bundle endowed with a projective unitary flat structure on it is called a {\it projective unitary flat $\mathbb P^{r}$-bundle}. If $P$ is a projective unitary flat $\mathbb P^{r}$-bundle and we replace $\tilde{\psi}_{\alpha}$ by $\psi_{\alpha}$ in the definition of the equivalence relation, then the maps $\varphi_{\alpha}$ define a {\it unitary flat automorphism of $P$}. Denote by $\operatorname{Aut}_{X}^{u}(P)$ the group of all unitary flat automorphisms of $P$. We call it the {\it unitary flat automorphism group} of $P$.

Looking at $P$ as a holomorphic $\mathbb{P}^{r}$-bundle, we define a {\it holomorphic automorphism of $P$} by a collection of maps $\phi_{\alpha}\colon U_{\alpha} \times \mathbb{P}^r \rightarrow U_{\alpha} \times \mathbb{P}^r$ such that
\begin{itemize}
  \item $\phi_{\alpha}(x, v)=\big( x, g_{\alpha}(x) \cdot v \big)$, where $g_{\alpha}\colon U_{\alpha} \rightarrow {\rm PGL}(r+1, \mathbb{C})$ is holomorphic.
  \item For any $\alpha$ and $\beta$, $\phi_{\alpha} \circ \psi_{\alpha} \circ \psi_{\beta} ^{-1} = \psi_{\alpha} \circ \psi_{\beta} ^{-1}\circ \phi_{\beta}$ on $\psi_{\beta}(P|_{U_{\alpha} \cap U_{\beta}})$.
\end{itemize}
We denote by $\operatorname{Aut}_{X}^{h}(P)$ the  holomorphic automorphism group of $P$. Then there is a natural inclusion $\operatorname{Aut}_{X}^{u}(P) \subset \operatorname{Aut}_{X}^{h}(P)$.

Similarly, we could define {\it unitary flat structures} on holomorphic vector bundles. In other words, all the unitary flat vector bundles of rank $r$ on $X$ constitute the set $H^1\big(X, U(r)\big)$.

\begin{defn}
\label{def:nonflat} {\rm (\cite{Mand73})}
{\rm Suppose that $P$ is a projective unitary flat $\mathbb{P}^1$-bundle on a Riemann surface $X$ of genus $g_{X} > 0$. Then we could choose a family of trivializations $\psi_{\alpha}\colon P|_{U_{\alpha}} \rightarrow U_{\alpha} \times \mathbb{P}^1$ such that the corresponding transition functions $g_{\alpha\beta}\colon U_{\alpha} \cap U_{\beta} \rightarrow {\rm PSU}(2)$ are  constant maps. For any holomorphic section $s$ of $P$, $\psi_{\alpha} \circ s|_{U_{\alpha}}$ can be viewed as a holomorphic map $s_{\alpha}\colon U_{\alpha} \rightarrow \mathbb{P}^1$. We call $s$ {\it locally flat} if $s_{\alpha}$ is a constant map for all $\alpha$.}
\end{defn}

Since the transition functions of $P$ are constant, a section $s$ is locally flat if and only if $s_{\alpha}$ is a constant map for some $\alpha$. This property does not depend on the choice of the projective unitary flat trivializations.

\begin{lem}
\label{lem:nonflat}
Let $E$ be a stable vector bundle of rank $2$ on a compact Riemann surface $X$ of genus $g_{X} \geq 1$. Hence, $\mathbb P(E)$ is a projective unitary flat $\mathbb P^{1}$-bundle. Then for any section $s \colon X \to \mathbb P(E)$, $s$ is {\rm not} locally flat.
\end{lem}
\begin{proof}
Since $E$ is a stable vector bundle on the Riemann surface $X$,  by the result of Narasimhan and Seshadri \cite{NS1965}, we know that $P = \mathbb P(E)$ arises from an irreducible projective unitary representation $\rho \colon \pi_{1}(X) \to {\rm PSU}(2)$. Hence, $\mathbb P(E)$ is a projective unitary flat $\mathbb P^{1}$-bundle. Let $\pi \colon \widetilde{X} \to X$ be the universal covering of $X$. Then $\pi^{*}(P)$ is a trivial $\mathbb P^{1}$-bundle on $\widetilde{X}$ and there is a canonical action of $\pi_{1}(X)$ on $\pi^{*}(P)\to \widetilde{X}$ such that the quotient  coincides with the original bundle $P\to X$. Moreover, if $s \in \Gamma(X, P)$ is locally flat, then $\pi^{*}(s)$ is a constant section of $\pi^{*}(P) \to \widetilde{X}$ which is invariant under the action of $\pi_{1}(X)$. Hence, the action of the group $\rho\big(\pi_{1}(X)\big)$ on $\mathbb P^{1}$ has at least one fixed point, which is in contradiction with the irreducibility of $\rho$.
\end{proof}

Now we could state the correspondence between cone spherical metrics representing effective divisors and projective unitary flat $\mathbb P^{1}$-bundles with sections which are not locally flat. 
\begin{thm}
\label{thm:correspondence}
Consider cone spherical metrics representing effective divisors on a compact Riemann surface $X$ of genus $g_{X} \geq 0$. Then there exists the following one-to-one correspondence
\[ \left \{ \begin{array}{c} \begin{tabular}{p{9em}}
Cone spherical metric $g$ representing an effective divisor.
\end{tabular}\end{array}\right \} \leftrightarrow
\left \{ \begin{array}{c} \begin{tabular}{p{13em}}
The pair $(P, s)$, where $P$ is a projective unitary flat $\mathbb P^{1}$-bundle and the section $s \colon X \to P$ is not locally flat.
\end{tabular}\end{array} \right \} \Big/ \operatorname{Aut}_{X}^{u}(P) \]
Note that $P = \mathbb P(E)$ for some rank two vector bundle $E$ since $X$ is a Riemann surface.
\begin{itemize}
\item If $g$ is an irreducible metric representing an effective divisor of even {\rm (}resp. odd{\rm )} degree, then $E$ is a stable vector bundle of even {\rm (}resp. odd{\rm )} degree.
\item If $g$ is reducible, then $E \cong (J \oplus J^{-1})\otimes L$ for some flat line bundle $J$ and some line bundle $L$ on $X$.
\end{itemize}
In summary, $E$ is a rank two polystable vector bundle on $X$.
\end{thm}

We prove this theorem in the former part of Subsection \ref{subsec:pf}.


\subsection{Branched projective coverings}
\label{subsec:branch}

In this subsection, we at first introduce, among others, the notions of branched projective coverings and indigenous bundles, which are crucial in the proof of Theorem \ref{thm:correspondence}. Then we prove Lemma \ref{lem:evendegree} for the proof of Theorems \ref{thm:correspondence} and \ref{thm:corr}.

Let $X$ be a compact Riemann surface, and $\{U_{\alpha}, z_{\alpha}\}$ a holomorphic coordinate covering of $X$. If for each $\alpha$, $w_{\alpha}\colon U_{\alpha} \rightarrow \mathbb{P}^1$ is not a constant holomorphic map such that $w_{\alpha}(p) = g_{\alpha \beta}(p) \cdot (w_{\beta}(p))$, where $g_{\alpha\beta}(p) \in {\rm PSL}(2, \mathbb{C})$ is independent of $p \in U_{\alpha} \cap U_{\beta}$, then $\{U_{\alpha}, w_{\alpha}\}$ is called a {\it branched projective covering} of $X$. Without loss of generality, we may assume that for each $\alpha$, $w_{\alpha}$ has at most one branch point $p_{\alpha}$ in $U_{\alpha}$ and $p_{\alpha} \not\in U_{\alpha} \cap U_{\beta}$ for all $\beta \neq \alpha$. Given a branched projective covering $\{U_{\alpha}, w_{\alpha}\}$ of $X$, we call the effective divisor
\[ B_{\{U_{\alpha},w_{\alpha}\}} = \sum_{p \in X} \nu_{w_{\alpha}}(p)\cdot p\]
the {\it ramification divisor} of  $\{U_{\alpha}, w_{\alpha}\}$, where $\nu_{w_{\alpha}}(p)$ is the branching order of $w_{\alpha}$ at the point $p$ \cite[Section 2]{Mand72}. Then we could naturally associate a flat $\mathbb{P}^1$-bundle $P$ on $X$ \cite[Section 2]{Gunning67} to the branched projective covering  $\{U_{\alpha},w_{\alpha}\}$ of $X$, and obtain canonically a section $s$ of $P$ defined by
\[ w_\alpha \colon U_\alpha\to {\mathbb P}^1\quad \text{ for all }\alpha, \]
which is not locally flat \cite[Section 2]{Gunning67}. Such a flat $\mathbb{P}^1$-bundle is called an {\it indigenous bundle} associated to the branched projective covering $\{U_{\alpha}, w_{\alpha}\}$ on $X$ \cite{Gunning67, Mand73}.

Motivated by \cite[Theorem 2]{Gunning67} and \cite[Theorem 2]{Mand73}, we find the following lemma.

\begin{lem}
\label{lem:evendegree}
Let $P$ be an indigenous bundle on a compact Riemann surface $X$ of genus $g_{X} \geq 0$ associated to a branched projective covering $\{U_{\alpha}, w_{\alpha}\}$. Then
\begin{enumerate}
\item If $P = \mathbb P(E)$ for some rank $2$ vector bundle on $X$ and $L$ is the line subbundle of $E$ which is the preimage of the section $s = \{w_{\alpha}\} \colon X \to P$, then $\mathcal O_{X} \left(B_{\{U_{\alpha},w_{\alpha}\}}\right) = L^{-2} \otimes \det E \otimes K_{X}$.
\item The ramification divisor $B_{\{U_{\alpha},w_{\alpha}\}}$ is of even degree if and only if $P = \mathbb{P}(E)$ for some flat rank two vector bundle $E$ with $\det E = \mathcal{O}_X$. Under this context,  there exists a meromorphic section ${\mathfrak s} = \{({\mathfrak s}_{1, \alpha}, {\mathfrak s}_{2,\alpha})\}$ of $E$ such that $w_{\alpha} = [{\mathfrak s}_{1,\alpha}:{\mathfrak s}_{2,\alpha}].$
\end{enumerate}
\end{lem}
\begin{proof} Since $\{U_{\alpha}, w_{\alpha}\}$ is a branched projective covering on $X$, we could choose a holomorphic coordinate covering  $\{U_{\alpha}, z_{\alpha}\}$  of $X$ with trivializations $\psi_{\alpha}\colon P|_{U_{\alpha}} \rightarrow U_{\alpha} \times \mathbb{P}^1$ and a family of matrices
$M_{\alpha\beta} = \begin{pmatrix}
   a_{\alpha\beta} & b_{\alpha\beta} \\
   c_{\alpha\beta} & d_{\alpha\beta}
\end{pmatrix} \in {\rm SL(2, \mathbb{C})}$ on intersections $U_{\alpha} \cap U_{\beta}\not=\emptyset$ such that
\begin{align}
\label{wtransition}
  w_{\alpha} = \frac{a_{\alpha\beta}w_{\beta} + b_{\alpha\beta}}{c_{\alpha\beta}w_{\beta} + d_{\alpha\beta}}
\end{align}
and $w_\alpha=w_{\alpha}(z_{\alpha})\colon U_\alpha\to {\mathbb C}$ are holomorphic functions by using suitable M{\" o}bius transformations if necessary. Hence we have
$\frac{dw_{\alpha}}{dw_{\beta}} = \frac{1}{(c_{\alpha\beta}w_{\beta}+d_{\alpha\beta})^2}$.
Since $K_X$ is defined by the transition functions
$k_{\alpha\beta} = \frac{dz_{\beta}}{dz_{\alpha}}$,
we have
\[ \lambda_{\alpha\beta} := (c_{\alpha\beta}w_{\beta}+d_{\alpha\beta})^2 = \frac{w_{\beta}'(z_{\beta})} {w_{\alpha}'(z_{\alpha})} \cdot \frac{dz_{\beta}}{dz_{\alpha}} = h_{\alpha\beta}k_{\alpha\beta}, \quad \text{where}\quad
h_{\alpha\beta} = \frac{w_{\beta}'(z_{\beta})} {w_{\alpha}'(z_{\alpha})}.\]
Let $H$ be the line bundle defined by the transition functions $h_{\alpha\beta}$. Then
$\{U_{\alpha}, \frac{1}{w_{\alpha}'(z_{\alpha})}\}$ forms a meromorphic section of $H$ so that
$H = \mathcal{O}_{X}\big(-B_{\{U_{\alpha}, w_{\alpha}\}}\big)$. Hence, for the first assertion, we only need to show $L^{2} = \Lambda \otimes \det E$, where $\Lambda$ is the line bundle defined by transition functions $\lambda_{\alpha \beta}$.

We choose a new holomorphic trivialization of $P|_{U_{\alpha}}$ as follows:
\begin{align*}
\phi_{\alpha} = \mu_{\alpha} \circ \psi_{\alpha} \colon P|_{U_{\alpha}} \xrightarrow{\psi_{\alpha}} U_{\alpha} \times \mathbb{P}^1 &\xrightarrow{\mu_{\alpha}} U_{\alpha} \times \mathbb{P}^1, \\
\text{where } \; \mu_{\alpha}(p, v) &= \left(p, \frac{1}{v - w_{\alpha}(p)}\right).
\end{align*}
Then, with respect to the new trivializations, the section $s=\{w_\alpha\}$ turns out to be $s = [1:0]$. Moreover,  we could obtain by computation that the corresponding transition functions have the form of
\[ \phi_{\alpha \beta}(p) \cdot v = (c_{\alpha \beta}w_{\beta} + d_{\alpha \beta})^{2} \ v + c_{\alpha \beta}(c_{\alpha \beta}w_{\beta} + d_{\alpha \beta}), \quad \forall p \in U_{\alpha} \cap U_{\beta}, v \in \mathbb{P}^{1}.\]
Actually, $\phi_{\alpha \beta}$ is defined by
\begin{align*}
\big(p, \, \phi_{\alpha \beta}(p) \cdot v \big) &= \phi_{\alpha} \circ \phi_{\beta}^{-1}(p, v) = \mu_{\alpha} \circ \psi_{\alpha} \circ \psi_{\beta}^{-1} \circ \mu_{\beta}^{-1}(p, v) \\
&= \mu_{\alpha} \circ (\psi_{\alpha} \circ \psi_{\beta}^{-1}) \left( p, \; w_{\beta} + \frac{1}{v} \right) \\
&= \mu_{\alpha} \left( p,\; \frac{(a_{\alpha \beta} w_{\beta} + b_{\alpha \beta}) \ v + a_{\alpha \beta}}{(c_{\alpha \beta} w_{\beta} + d_{\alpha \beta}) \ v + c_{\alpha \beta}} \right) \\
&= \left( p, \; \frac{1}{\frac{(a_{\alpha \beta} w_{\beta} + b_{\alpha \beta}) \ v + a_{\alpha \beta}}{(c_{\alpha \beta} w_{\beta} + d_{\alpha \beta}) \ v + c_{\alpha \beta}} - w_{\alpha} }\right) \\
&= \left( p, \; \frac{(c_{\alpha \beta} w_{\beta} + d_{\alpha \beta}) \ v + c_{\alpha \beta}}{a_{\alpha \beta} - c_{\alpha \beta} w_{\alpha}} \right) \quad \text{since} \; w_{\alpha} = \frac{a_{\alpha\beta}w_{\beta} + b_{\alpha\beta}}{c_{\alpha\beta}w_{\beta} + d_{\alpha\beta}} \\
&= \left( p, \; \frac{ \big( (c_{\alpha \beta} w_{\beta} + d_{\alpha \beta}) \ v + c_{\alpha \beta} \big) (c_{\alpha \beta} w_{\beta} + d_{\alpha \beta})}{a_{\alpha \beta} (c_{\alpha \beta} w_{\beta} + d_{\alpha \beta}) - c_{\alpha \beta} (a_{\alpha\beta}w_{\beta} + b_{\alpha\beta})} \right) \\
&= \left( p, \; (c_{\alpha \beta}w_{\beta} + d_{\alpha \beta})^{2} \ v + c_{\alpha \beta}(c_{\alpha \beta}w_{\beta} + d_{\alpha \beta}) \right).
\end{align*}

On the other hand,  let $\ell_{\alpha \beta}$ be the transition functions of $L$. Then the transition functions of $E$ look like
\[ \begin{pmatrix}
  \ell_{\alpha \beta} & t_{\alpha \beta} \\
   0 & \ell_{\alpha \beta}^{-1} \xi_{\alpha \beta}
\end{pmatrix}, \]
where $\xi = \det E$. Hence, $g_{\alpha \beta}(p) \cdot v = \ell_{\alpha \beta}^{2} \xi^{-1}_{\alpha \beta} \ v + \ell_{\alpha \beta}\xi^{-1}_{\alpha \beta} t_{\alpha \beta}$ is a transition function of $P$. Moreover, $s = [1:0]$ since it is defined by $L \to E$.

Note that the $\mathbb P^{1}$-bundles defined by $\phi_{\alpha \beta}$ and $g_{\alpha \beta}$ are holomorphically equivalent. Let $\varphi_{\alpha}\colon U_{\alpha} \times \mathbb P^{1} \rightarrow U_{\alpha} \times \mathbb P^{1}$ be a  1-coboundary with $\varphi_{\alpha}(p, v) = \big(p, g_{\alpha}(p) \cdot v \big)$ \cite[p.409]{Atiyah55}. Then $g_{\alpha} \cdot \big( [1:0] \big) = [1:0]$, i.e. $g_{\alpha}$ is an affine transformation. Let $g_{\alpha} \cdot v = A_{\alpha}\, v + B_{\alpha}\ (A_{\alpha} \neq 0, v \in \mathbb P^{1} \setminus \{ \infty \})$, by
\[ \phi_{\alpha \beta} = g_{\alpha} \circ g_{\alpha \beta} \circ g_{\beta}^{-1} \ \text{ on } U_{\alpha} \cap U_{\beta}, \]
we obtain that $(c_{\alpha \beta}w_{\beta} + d_{\alpha \beta})^{2} = A_{\alpha} \, \ell_{\alpha \beta}^{2} \xi^{-1}_{\alpha \beta} \, A_{\beta}^{-1}$. Therefore, line bundles defined by $(c_{\alpha \beta}w_{\beta} + d_{\alpha \beta})^{2}$ and $\ell_{\alpha \beta}^{2} \xi^{-1}_{\alpha \beta}$ are equivalent, i.e. $\Lambda = L^{2} \otimes (\det E)^{-1}$.

For the second part, note that $H$ has degree $\left(-\deg B_{\{U_{\alpha}, w_{\alpha}\}}\right)$ and the line bundle $\Lambda$ has degree $\left(2g_{X} - 2 - \deg B_{\{U_{\alpha}, w_{\alpha}\}}\right)$.

Suppose that $\deg B_{\{U_{\alpha}, w_{\alpha}\}}$ is even. Then so is $\deg \Lambda$. There exists a line bundle $N$ defined by $N_{\alpha\beta}$ such that $N_{\alpha\beta}^2 = \lambda_{\alpha\beta}$. By changing the sign of $M_{\alpha\beta}\in {\rm SL(2,\,{\mathbb C})}$ if necessary,  we have $N_{\alpha\beta} = c_{\alpha\beta}w_{\beta} + d_{\alpha\beta}$. Taking a non-trivial  meromorphic section $\{U_{\alpha}, f_{\alpha}\}$ of $N$, we find
\[ \begin{pmatrix}
  w_{\alpha}f_{\alpha} \\
  f_{\alpha}
\end{pmatrix} = \begin{pmatrix}
  a_{\alpha\beta} & b_{\alpha\beta} \\
  c_{\alpha\beta} & d_{\alpha\beta}
\end{pmatrix} \begin{pmatrix}
  w_{\beta}f_{\beta} \\
  f_{\beta}
\end{pmatrix} \]
and $M_{\alpha\beta} = M_{\alpha\gamma} M_{\gamma\beta}$. The desired flat rank two vector bundle $E$ is just the one defined by $M_{\alpha\beta}$. Furthermore,  $(w_{\alpha}f_{\alpha}, f_{\alpha})$ forms a meromorphic section ${\mathfrak s}=({\mathfrak s}_{1,\alpha},\,{\mathfrak s}_{2,\alpha})$ of $E$ satisfying $w_\alpha=[{\mathfrak s}_{1,\alpha} : {\mathfrak s}_{2,\alpha}]$.

Suppose that there exists a rank two flat vector bundle $E$ with ${\mathbb P}(E)=P$ for the indigenous bundle $P$ which is associated to the projective covering $\{(U_\alpha,\,w_\alpha)\}$. Then $s = \{w_\alpha\}$ is the canonical section of $P$. By the argument in the first paragraph of \cite[Section 4]{Atiyah57II}, there exists a line subbundle $L$ of $E$ generated by $s$ such that each non-trivial meromorphic section ${\mathfrak s}=\big({\mathfrak s}_{1,\alpha},\,{\mathfrak s}_{2,\,\alpha}\big)$ of the line bundle $L\subset E$ satisfies $w_\alpha=[{\mathfrak s}_{1,\alpha} : {\mathfrak s}_{2,\alpha}]$. Denote by
$ M_{\alpha\beta} = \begin{pmatrix}
   a_{\alpha\beta} & b_{\alpha\beta} \\
   c_{\alpha\beta} & d_{\alpha\beta}
\end{pmatrix} \in {\rm SL(2, \mathbb{C})}$
the transition functions of $E$. Recalling \eqref{wtransition} in the first paragraph of the proof, we only need to show the degree of the line bundle $\Lambda$ is even since $\Lambda = H \otimes K_X$ and
\[ \deg\, B_{\{U_\alpha,\,w_\alpha\}}\equiv \deg\, H\quad (\text{mod}\ 2).\]
Since
$w_\alpha = [{\mathfrak s}_{1,\alpha} :{\mathfrak s}_{2,\alpha}]$, we have
\[ \begin{pmatrix}
  w_{\alpha}{\mathfrak s}_{2,\alpha} \\
  {\mathfrak s}_{2,\alpha}
\end{pmatrix} = \begin{pmatrix}
  {\mathfrak s}_{1,\alpha} \\
  {\mathfrak s}_{2,\alpha}
\end{pmatrix} = \begin{pmatrix}
  a_{\alpha\beta} & b_{\alpha\beta} \\
  c_{\alpha\beta} & d_{\alpha\beta}
\end{pmatrix} \begin{pmatrix}
 {\mathfrak s}_{1,\beta} \\
  {\mathfrak s}_{2,\beta}
\end{pmatrix} = \begin{pmatrix}
  a_{\alpha\beta} & b_{\alpha\beta} \\
  c_{\alpha\beta} & d_{\alpha\beta}
\end{pmatrix} \begin{pmatrix}
  w_{\beta}{\mathfrak s}_{2,\beta} \\
  {\mathfrak s}_{2,\beta}
\end{pmatrix} \]
and ${\mathfrak s}_{2, \alpha} = (c_{\alpha\beta}w_{\beta} + d_{\alpha\beta}) {\mathfrak s}_{2,\beta}$. It follows that the functions $\left(c_{\alpha\beta}w_{\beta} + d_{\alpha\beta}\right)$ on $U_{\alpha} \cap U_{\beta}$ are 1-cocyles, which define a line bundle $N$ with $N^2 = \Lambda$. Hence, $\deg \Lambda$ is even.
\end{proof}

Let $g$ be a cone spherical metric on $X$ representing an effective divisor $D = \sum_{j=1}^n \beta_j \, p_j$. Suppose that $\{U_{\alpha}, z_{\alpha}\}$ is a holomorphic coordinate covering of $X$ such that each $U_\alpha$ is a sufficiently small  disc and contains at most one point in ${\rm Supp}\,D$. Then there exists a holomorphic map $f_{\alpha}\colon U_{\alpha} \rightarrow \mathbb{P}^1$ for each $\alpha$ such that it has at most one ramified point and $g|_{U_{\alpha}} = f_{\alpha}^*\,g_{\rm st}$ \cite[Lemmas 2.1 and 3.2]{CWWX2015}, where $g_{\rm st}=\frac{4|dw|^2}{(1+|w|^2)^2}$ is the standard metric on ${\mathbb P}^1$. Then these pairs $\{U_{\alpha}, f_{\alpha}\}$ define a branched projective covering of $X$ with ramification divisor being $D$ such that
  \[ f_{\alpha}(p) = f_{\alpha \beta}(p) \cdot f_{\beta}(p),  \]
where $f_{\alpha \beta}(p) \in {\rm PSU(2)} \subseteq {\rm PSL}(2, \mathbb C)$ is  independent of $p \in U_{\alpha} \cap U_{\beta}$. Then we obtain an indigenous bundle $P$ defined by $f_{\alpha \beta}$ on $X$ associated to the branched projective covering $\{U_\alpha,\, f_{\alpha}\}$ and a canonical section $s=\{f_{\alpha}\}$ of $P$ which is not locally flat.

As an application of Lemma \ref{lem:evendegree}, we could give a criterion for the parity of the degree of effective divisors represented by cone spherical metrics as follows.

\begin{cor}
\label{cor:criterion}
Let $D$ be an effective divisor represented by a cone spherical metric $g$ on a compact Riemann surface $X$ of genus $g_X\geq 0$. Then $\deg D$ is even if and only if the monodromy representation $\rho_f\colon \pi_1(X)\to {\rm PSU(2)}$ of a developing map $f$ of $g$ could be lifted to $\widetilde{\rho_f}\colon \pi_1(X)\to {\rm SU}(2)$ with the following commutative diagram
\[ \xymatrix{
  \pi_{1}(X) \ar[r]^{\widetilde{\rho_f}} \ar[rd]_{\rho _{f}} & {\rm SU(2)} \ar[d] ^{\pi} \\
                                                                                & {\rm PSU(2)}.
} \]
\end{cor}

\subsection{Proof of Theorems \ref{thm:correspondence} and \ref{thm:corr}}
\label{subsec:pf}
With the help of Lemma \ref{lem:evendegree}, we could now prove Theorem \ref{thm:correspondence}.
\begin{proof}[Proof of Theorem \ref{thm:correspondence}]

Suppose that $g$ is a cone spherical metric on $X$ representing an effective divisor $D = \sum_{j=1}^n \beta_j \, p_j$. Then, by the argument before Corollary \ref{cor:criterion}, we could obtain a branched projective covering $\{U_{\alpha}, f_{\alpha}\}$ with ramification divisor $D$ such that $g|_{U_\alpha}=f^*_\alpha\, g_{\rm st}$ and an indigenous bundle $P$ associated to this covering such that the transition functions $f_{\alpha\beta}$ lie in ${\rm PSU(2)}$. Moreover, $\{U_{\alpha}, f_{\alpha}\}$ defines a section $s$ of $P$ which is not locally flat.

Then we show that the equivalence class of $(P, s)$ does \emph{not} depend on the choice of $\{U_{\alpha}, f_{\alpha}\}$. At first, we know that the monodromy representation $\rho \colon \pi_{1}(X) \to {\rm PSU}(2)$ of a developing map of $g$ is unique up to conjugation. By the one-to-one correspondence between the set of representations of $\pi_{1}(X)$ and that of flat fiber bundles, we conclude that $P$ is unique. If $\{U_{\alpha}, f_{1, \alpha}\}$ and $\{U_{\alpha}, f_{2, \alpha}\}$ are two sections of $P$ obtained by $g$, then $f_{1,\alpha}^*\, g_{\rm st} = g|_{U_\alpha}= f_{2,\alpha}^*\, g_{\rm st}$. Hence \cite[Section 2]{CWWX2015}
\[ f_{1, \alpha} = T_{\alpha} \cdot f_{2, \alpha} = \frac{a_{\alpha}f_{2, \alpha} + b_{\alpha} }{-\bar{b}_{\alpha} f_{2, \alpha} + \bar{a}_{\alpha}}, \; \text{where }a_{\alpha}, b_{\alpha} \in \mathbb{C}\text{ and }|a_{\alpha}|^{2} + |b_{\alpha}|^{2} = 1. \]
Then we could obtain that
\[ (T_{\alpha} \circ f_{\alpha \beta})(f_{2, \beta}) = (f_{\alpha \beta} \circ T_{\beta})(f_{2, \beta}) \]
on $U_{\alpha} \cap U_{\beta}$. Note that $f_{2, \beta}$ is not a constant map. We have
\[ T_{\alpha} \circ f_{\alpha \beta} = f_{\alpha \beta} \circ T_{\beta}. \]
That is, $\{U_{\alpha}, f_{1, \alpha}\}$ and $\{U_{\alpha}, f_{2, \alpha}\}$  only differ by a unitary flat automorphism of $P$ defined by $\{T_{\alpha}\}$.

Let $(P, s)$ be a pair on the right-hand side of the correspondence, and let $\{U_{\alpha}, s_{\alpha}\}$ be the local expression of the section $s$. Then on any intersection $U_{\alpha} \cap U_{\beta}$, the local expressions satisfy $ s_{\alpha} = f_{\alpha\beta} \cdot s_{\beta}$, where the transition functions $f_{\alpha\beta}$ lie in ${\rm PSU}(2)$. Defining $g_{\alpha} = s_{\alpha}^*g_{\rm st}$ on each $U_{\alpha}$, we find $g_{\alpha}|_{U_{\alpha}\cap U_{\beta}} = g_{\beta}|_{U_{\alpha}\cap U_{\beta}}$ since $f_{\alpha\beta} \in {\rm PSU(2)}$. Hence, the metrics $g_{\alpha}$ on $U_\alpha$ define a global cone spherical metric $g = s^*g_{\rm st}$ representing an effective divisor on $X$. It is obvious that if $(P, s_{1})$ and $(P, s_{2})$ are equivalent, then $s_{1}^*g_{\rm st}= s_{2}^*g_{\rm st}$.

Therefore, by the above two paragraphs, we complete the proof of the one-to-one correspondence in the theorem.

Since $X$ is a compact Riemann surface, there exists a rank two holomorphic vector bundle $E$ on $X$ such that $P = \mathbb{P}(E)$. By the results of Narasimhan and Seshadri \cite{NS1965}, we know that $E$ is stable if $g$ is irreducible (so is $\rho$). By Lemma \ref{lem:evendegree}, we conclude that the degree of  the ramification divisor represented by $g$ and $\deg E$ have the same parity.

At last, suppose $g$ is a reducible metric on $X$. Then the monodromy representation $\rho \colon \pi_1(X) \rightarrow {\rm PSU(2)}$ of $P$ can be lifted to $\tilde{\rho}\colon \pi_1(X) \rightarrow {\rm SU}(2)$. By Lemma 4.1 in \cite{CWWX2015}, we have,  up to a conjugation,
\[ \tilde{\rho}(\pi_1(X)) \subset \left\{\text{diag}(e^{\sqrt{-1} \; \theta}, e^{ -\sqrt{-1} \; \theta}) \mid \theta \in \mathbb{R}\right\} \subset {\rm SU(2)}. \]
It follows that there exist two unitary flat line bundles $J_1$ and $J_2$ on $X$ such that $E = J_1 \oplus J_2$ and $J_1 = J_2^{-1}$. We  complete the proof of Theorem \ref{thm:correspondence}.
\end{proof}
From the proof of Theorem \ref{thm:correspondence}, we know that
\begin{rem}
\label{rem:divisorequ}
If $(P, s)$ is a pair in Theorem \ref{thm:correspondence}, then $P$ is an indigenous bundle on $X$. Moreover, the effective divisor represented by $g$ corresponding to $(P, s)$ coincides with the ramification divisor of the branched projective covering associated to $(P, s)$.
\end{rem}
In order to complete the proof of Theorem \ref{thm:corr}, we need the following algebraic fact which is motivated by a note of A. Beauville \cite{Beau2000}.
\begin{prop}
\label{prop:directimage}
Let $X$ be a compact Riemann surface of genus $g_X\geq 1$ and $T$ an $r$-torsion line bundle on $X$. Let $\pi \colon \widehat{X} \to X$ be the degree $r$ cyclic \'etale covering associated to $T$. Then a rank $r$ vector bundle $E$ on $X$ satisfies $E \cong T \otimes E $ if and only if $E\cong \pi_{*}L$ for some line bundle $L$ on $\widehat{X}$.
\end{prop}
\begin{proof}
If we define
\[ \mathcal{L} = \mathcal{O}_{X} \oplus T \oplus T^{2} \oplus \cdots \oplus T^{r-1}, \]
then $\mathcal{L}$ is a sheaf of $\mathcal{O}_{X}$-algebra under the isomorphism $T^{r} \cong \mathcal{O}_{X}$ and
\[ \pi \colon \widehat{X} = \operatorname{Spec }\mathcal{L} \to \operatorname{Spec } \mathcal{O}_{X} = X \]
is the cyclic \'etale covering associated to $T$ together with $\pi_{*} \mathcal{O}_{\widehat{X}} \cong  \mathcal{L}$.

Let $L$ be a line bundle on $\widehat{X}$. We have
\begin{align*}
T \otimes_{\mathcal{O}_{X}} \pi_{*}L &\cong T \otimes_{\mathcal{O}_{X}} \pi_{*} \mathcal{O}_{\widehat{X}} \otimes_{\pi_{*} \mathcal{O}_{\widehat{X}}} \pi_{*}L \\
&\cong (T \otimes_{\mathcal{O}_{X}} \mathcal{L}) \otimes_{\pi_{*} \mathcal{O}_{\widehat{X}}} \pi_{*}L \\
&\cong \mathcal{L} \otimes_{\pi_{*} \mathcal{O}_{\widehat{X}}} \pi_{*}L \\
&\cong \pi_{*}L.
\end{align*}

On the other hand, if $E$ is a vector bundle of rank $r$ with $T \otimes_{\mathcal{O}_{X}} E \cong E$ on $X$, then we fix an isomorphism $f \colon T \otimes_{\mathcal{O}_{X}} E \to E$ and define
\begin{align*}
f^{(0)} &\colon \mathcal{O}_{X} \otimes_{\mathcal{O}_{X}} E \to E, \\
f^{(1)} &\colon T \otimes_{\mathcal{O}_{X}} E \xrightarrow{f} E, \\
f^{(i)} &\colon T^{i} \otimes_{\mathcal{O}_{X}} E \cong T \otimes_{\mathcal{O}_{X}} (T^{i-1} \otimes_{\mathcal{O}_{X}} E ) \xrightarrow{\operatorname{Id}_{T} \otimes f^{(i-1)}} T \otimes_{\mathcal{O}_{X}} E \xrightarrow{f} E.
\end{align*}
Hence $f^{(i)} \, (0 \leq i \leq r-1)$ are isomorphisms. Now we could define a homomorphism of $\mathcal{O}_{X}$-modules
\begin{align*}
\widetilde{f} \colon \mathcal{L} \otimes_{\mathcal{O}_{X}} E &\to E, \\
(a_{0}, a_{1}, \cdots, a_{r-1}) \otimes s &\mapsto \sum_{i=0}^{r-1} f^{(i)}(a_{i} \otimes s),
\end{align*}
where $s$ is a section of $E$ and $a_{i}$ is a section of $T^{(i)}$ for $0 \leq i < r$. This homomorphism endows $E$ with an $\mathcal{L}$-module structure. Denote this $\mathcal{L}$-module by $L$ which could be viewed as a sheaf on $\widehat{X}$. Then we have $\pi_{*}L \cong E$. It is obvious that $L$ is of rank one if we could prove it is a locally free $\mathcal{O}_{\widehat{X}}$-module.

Let $S$ be the maximal torsion subsheaf of $L$. Note that $\widehat{X}$ is a compact Riemann surface. Hence, there exist $p_{1}, p_{2}, \cdots, p_{N} \in \widehat{X}$ such that
\[ S \cong \bigoplus_{i=1}^{N} \mathbb{C}_{p_{i}},  \]
where $\mathbb{C}_{p_{i}}$ is a sky-scraper sheaf on $\widehat{X}$. Then $\pi_{*}S = \bigoplus_{i=1}^{N}\mathbb{C}_{\pi(p_{i})}$ is a subsheaf of $\pi_{*}L \cong E$. We obtain that $N = 0$ since $E$ is a locally free $\mathcal{O}_{X}$-module. Therefore $L$ is torsion-free, which implies it is locally free.
\end{proof}

At the very end of this section, we give the proof of Theorem \ref{thm:corr} in what follows, where we use the language about stable embeddings.  Two embeddings $L \xrightarrow{i} E$ and $L' \xrightarrow{i'} E'$ are called {\it equivalent} if there exists a line bundle $N$ on $X$ with the following commutative diagram
\[ \xymatrix{
L \otimes N \ar@{=}[d] \ar[r]^{i \otimes \operatorname{id}_{N}} & E \otimes N \ar@{=}[d] \\
L'                \ar[r]^{i'}                   & E'.
} \]

\begin{proof}[Proof of Theorem \ref{thm:corr}]
Let $E$ be a vector bundle of rank $2$ on a Riemann surface $X$. Then we can think of a section $s \colon X \to \mathbb{P}(E)$ as a line subbundle of $E$ by taking the preimage of $s$ in $E$. On the other hand, if $L$ is a line subbundle of $E$, then $L$ defines a section $s_{(L, E)}$ of $\mathbb{P}(E)$ by $x \mapsto \left[L|_{\{x\}}\right] \in \mathbb{P}(E)$. Moreover, two equivalent embeddings $(L, E)$ and $(L',\, E')$ give the same pair $(P,\,s)$.
Suppose that $E$ is a rank two stable vector bundle on $X$ and $L$ is a line subbundle of $E$. Then the projective bundle $\mathbb{P}(E)$ is a projective unitary flat $\mathbb{P}^1$-bundle. The section $s_{(L, E)}$ of $\mathbb P(E)$ is not locally flat by Lemma \ref{lem:nonflat}. According the correspondence of Theorem \ref{thm:correspondence}, we could construct an irreducible metric from a stable embedding $L \to E$, and all irreducible metrics representing effective divisors arise in this way. Hence we obtain a canonical surjective map from $\mathcal{SE}(X)$ to $\mathcal{MI}(X,\,{\Bbb Z})$, denoted by $\sigma$.

\begin{enumerate}

\item Let $D$ be the effective divisor represented by the cone spherical metric given by an embedding $L \to E$. Recall that $(\mathbb P(E), s_{(L, E)})$ defines a branched projective covering on $X$ whose ramification divisor coincides with $D$. Hence, by the first conclusion in Lemma \ref{lem:evendegree}, we obtain that $D$ lies in the complete linear system $\left| L^{-2} \otimes \det E \otimes K_{X} \right|$, from which the first assertion about $\sigma$ holds.

\item Then we prove the second assertion of $\sigma$ in Theorem \ref{thm:corr}. Actually, there exists a one-to-one correspondence between equivalence classes of embeddings $L \to E$ and pairs $\big({\Bbb P}(E),\,s\big)$ consisting of a $\mathbb{P}^{1}$-bundle $\mathbb{P}(E)$ and a section $s$ of ${\Bbb P}(E)$.
 Choose an irreducible metric $g$ representing an effective divisor and take a stable embedding $L\to E$ defining $g$.
 Hence, by Theorem \ref{thm:correspondence}, the preimage $\sigma^{-1}(g)$ of $g$ coincides with the orbit through the section $s_{(L,E)}$  with respect to the  free $\operatorname{Aut}_{X}^{u}(\mathbb{P}(E))$-action on the space $\Gamma\big(X, {\Bbb P}(E)\bigr)$ of sections of ${\Bbb P}(E)$, which will be explicitly given in next paragraph.
 A result of Grothendieck \cite[Section 5]{Gro1958} says that there exists an exact sequence of groups
\[ 1 \to \operatorname{Aut}_{X}^{h}(E)/ \Gamma(X, \mathcal{O}_{X}^{*}) \to \operatorname{Aut}_{X}^{h} (\mathbb{P}(E)) \to \varXi \to 1, \]
where $\varXi = \left\{ T \in H^{1}(X, \mathcal{O}_{X}^{*}) \mid E \cong E \otimes T \right\}$  is a subgroup of the group of $2$-torsions in ${\rm Pic}^{0}(X)$. Since $X$ is a compact Riemann surface and $E$ is stable, we obtain
$\operatorname{Aut}_{X}^{h} (\mathbb{P}(E)) \cong \varXi.$
Hence we have
\[ \left| \operatorname{Aut}_{X}^{u} (\mathbb{P}(E)) \right| \leq \left| \operatorname{Aut}_{X}^{h} (\mathbb{P}(E)) \right| \leq 2^{2g_{X}}, \]
which implies $|\sigma^{-1}(g)|\leq 2^{2g_X}$. Hence we complete the proof of the second assertion.

Moreover, we choose $T$ in $\operatorname{Aut}_{X}^{u} (\mathbb{P}(E))$ and denote the isomorphism by $\eta \colon E \cong E \otimes T$. Then, under the action of $T$, the embedding $L \to E$ is mapped to another one given by $\eta(L) \otimes T^{-1} \to E$, which shows that the $\operatorname{Aut}_{X}^{u} (\mathbb{P}(E))$-action on $\Gamma\big({\Bbb P}(E)\big)$ is free.
If $E \cong E \otimes T$ for some $T \in \varXi\backslash\{e\}$, by Proposition \ref{prop:directimage},
$E$ must be the direct image of some line bundle over $\widetilde X$ by the double et\' ale covering $\pi_T:\widetilde X\to X$ associated to $T$. Hence, the moduli space of such $E$'s has dimension $\leq g_{\widetilde X}=2g_{X} - 1$. Therefore, for a generic stable vector bundle $E$ on $X$ of genus $g_{X} \geq 2$, we have $\operatorname{Aut}_{X}^{u} (\mathbb{P}(E)) \subset \operatorname{Aut}_{X}^{h} (\mathbb{P}(E)) =\{e\}$, i.e. $\operatorname{Aut}_{X}^{u} (\mathbb{P}(E)) =\{e\}$.

\item To show the third assertion of $\sigma$, we firstly make

\vspace{5pt}
{\bfseries Claim 1:} {\it The equivalence class of a stable embedding $L \to E$ is {\it generic} in the stratum $\mathcal{SE}_k(X)$ if $E$ is a generic stable vector bundle of rank two and the positive index $k=\deg\,E - 2\deg\, L > 10g_{X} - 5$}.
\vspace{5pt}

Actually, $E$ is indecomposable since it is stable. Then, by a theorem of Atiyah \cite[Theorem 1]{Atiyah57II}, we know that there exists a positive integer
\[ N = - \frac{\deg E}{2} + 5g_{X} - 2, \]
such that for any line bundle $L$ with $\deg L \leq - N$ and any rank $2$ stable vector bundle $E$, $L^{-1} \otimes E$ is {\it ample in the sense of Atiyah} \cite[p.417]{Atiyah57II}, i.e. {\it $L^{-1} \otimes E$ is globally generated and $ H^{q}(X, L^{-1} \otimes E) = 0$ for all $q > 0$}. Moreover, $\mathcal{O}_{X}$ is a line subbundle of $L^{-1} \otimes E$ \cite[Theorem 2]{Atiyah57II}, i.e. there exists an embedding $L\to E$ if $k = \deg E - 2 \deg L > 10g_{X} - 5$. Hence we proved the claim.

Since $L^{-1} \otimes E$ is ample in the sense of Atiyah op.cit., we have $ H^{q}(X, L^{-1} \otimes E) = 0$ for all $q > 0$. By the Riemann-Roch formula, we find that
\[ \dim H^{0}(X, L^{-1} \otimes E) = 2(1 - g_{X}) + (\deg E - 2 \deg L), \]
monotonically increases to infinity as the index $k=(\deg E - 2 \deg L)\to\infty$.
If $L^{-1} \otimes E$ is ample, then we make

\vspace{5pt}
{\bfseries Claim 2}: {\it If $s \in H^{0}(X, L^{-1} \otimes E)$ is a generic section, then $s$ is nowhere vanishing and gives a stable embedding $L\to E$.}
\vspace{5pt}

In fact, since $L^{-1}\otimes E$ is ample in the sense of Atiyah op.cit., we have that $L^{-1}\otimes E$ is globally generated, i.e. $$H^{0}(X, L^{-1} \otimes E) \xrightarrow{\operatorname{ev}_{x}} E|_{\{x\}},\; s \mapsto s(x)$$ is an epimorphism for each $x \in X$. Let $K_{x} = \ker (\operatorname{ev}_{x})$, and let $S$ be the subvariety of  $H^{0}(X, L^{-1} \otimes E)$ generated by all $K_{x}$'s. Then we have
\begin{eqnarray*}
\dim S &\leq& \Big(\dim H^{0}(X, L^{-1} \otimes E) - \dim E|_{\{x\}}\Big) + \dim X \\&=&
 \dim H^{0}(X, L^{-1} \otimes E) - 1.
 \end{eqnarray*}
Hence, if $s \in H^{0}(X, L^{-1} \otimes E)$ is a generic section, then $s \notin S$, i.e. $s$ is a nowhere vanishing section, which induces an embedding of $L$ into $E$ since $\mathcal{H}om(L,\,E)\cong L^{-1}\otimes E$.
Hence we complete the proof of the second claim.

By the first claim,  the forgetful map $(L\to E)\mapsto E$ gives a surjective map from ${\mathcal {SE}}_k(X)$ onto the moduli space of rank two stable vector bundles if the index $k > 10g_{X} - 5$. Hence, as the index $k\to\infty$, the dimension of ${\mathcal {SE}}_k(X)$  monotonically increases to infinity.  Recall that for a generic stable vector bundle $E$ on $X$ of genus $g_{X} \geq 2$, we have that $\operatorname{Aut}_{X}^{u} (\mathbb{P}(E)) =\operatorname{Aut}_{X}^{h} (\mathbb{P}(E))=\{e\}$, which implies that an irreducible metric $g$ given by a stable embedding $L\to E$ has a unique pre-image, i.e. there exists a unique equivalence class of the stable extension $0\to L\to E\to E/L\to 0$ defining $g$.
Summing up this and the preceding two claims, we find that, as $k > 10g_{X} - 5$, the restriction of $\sigma$ to some Zariski open subset
of ${\mathcal {SE}}_k(X)$  is injective.
Therefore, we complete the proof of the theorem.

\end{enumerate}
\end{proof}

\section{A Lange-type theorem}
\label{sec:langetype}
A part of Theorem \ref{thm:stable} is contained in \cite[Proposition 1.1]{LN83} which Lange and Narasimhan proved by counting the dimensions of secant varieties. We give a self-contianed proof for the whole of the theorem in this section.

Let $L$ be a line bundle of degree $d > 2$ on a compact Riemann surface $X$ of genus $g_{X} \geq 1$. Let $\mathcal{R} = \mathbb{P} \big( \operatorname{Ext}_{X}^{1}(L, \mathcal O_{X})\big)$ be the variety parameterizing the equivalence classes of nontrivial extensions of $L$ by $\mathcal O_{X}$. By the Serre duality and the Riemann-Roch theorem, we have
\begin{align*}
\dim \mathcal{R} &= \dim \operatorname{Ext}_{X}^{1}(L, \mathcal O_{X}) -1 \\
&= \dim H^{1}(X, L^{-1}) - 1 \\
&= d + g_{X} - 2.
\end{align*}
On the other hand, for any given nontrivial extension of $L$ by $\mathcal O_{X}$
\[ {\Bbb E}:  0 \rightarrow \mathcal O_{X} \rightarrow E \rightarrow L \rightarrow 0. \]
We obtain a nonzero element ${\Bbb E} \in \operatorname{Ext}_{X}^{1}(L, \mathcal O_{X}) \cong H^{1}(X, L^{-1})$. Consider the variety of hyperplanes in $H^{0}(X, K_{X} \otimes L)$, which is isomorphic to $\mathbb{P}^{N}$ with $N =(d + g_{X} - 2)$.  By Serre duality, the annihilating space of $\mathbb E$ is a hyperplane in $H^{0}(X, K_{X} \otimes L)$,  we denote it by $[{\Bbb E}] \in \mathbb{P}^{N}$. Then we obtain a one-to-one correspondence between $\mathcal{R}$ and $\mathbb{P}^{N}$.

For each $p \in X$, we have
\[ h^{1} \big(X, K_{X} \otimes L \otimes \mathcal{O}_{X}(-p)\big) = h^{0}\big(X, L^{-1} \otimes \mathcal{O}_{X}(p)\big) = 0. \]
Then the following map
\[ \phi \colon X \rightarrow \mathbb{P}^{N}, \quad \phi(p) = \left\{ \sigma \in H^{0}(X, K_{X} \otimes L) \mid \sigma(p) = 0\right\}, \]
is well defined. Moreover, $\phi$ is an embedding since
\[ h^{1}\big(X, K_{X} \otimes L \otimes \mathcal{O}_{X}(-p - q)\big) = 0,\; \forall p,q \in X. \]
Let Sec$_{k}(X)$ be the $k^{th}$ secant variety of $\phi(X)$,  i.e. the Zariski closure of the union of the linear spaces spanned by $k$ generic points in $\phi(X)$. Hence, $\dim \operatorname{Sec}_{k}(X) \leq 2k - 1$. In particular, Sec$_{1}(X) = \phi(X)$.

\begin{lem}
\label{lem:secant}
Let $X$ be a compact Riemann surface of genus $g_{X} \geq 1$, and $L$ a line bundle of degree $d > 2$. Suppose that
${\Bbb E}:  0 \rightarrow \mathcal O_{X} \rightarrow E \rightarrow L \rightarrow 0$
is a nontrivial extension of $L$ by $\mathcal O_{X}$. Then $[{\Bbb E}] \in \operatorname{Sec}_{k}(X)$ if and only if there exist $k$ points $p_{1}, p_{2}, \cdots, p_{k}$ in $X$ such that ${\Bbb E}$ is contained in the kernel of the map
\[ H^{1}(X, L^{-1}) \rightarrow H^{1}\big(L^{-1} \otimes \mathcal{O}_{X}(p_{1} + \cdots + p_{k})\big), \]
which is induced from the natural map $ H^{0}\big(X, K_{X} \otimes L \otimes \mathcal{O}_{X}(-p_{1} - \cdots -p_{k}) \big) \rightarrow H^{0}\big(X, K_{X}\otimes L \big)$ by Serre duality.
\end{lem}
\begin{proof}
If $[{\Bbb E}] \in \operatorname{Sec}_{k}(X)$, then by the definition of $\operatorname{Sec}_{k}(X)$ we know that there exist $k$ points $p_{1}, \cdots, p_{k} \in X$ such that the subspace $\phi(p_{1}) \cap \phi(p_{2}) \cap \cdots \cap \phi(p_{k})$ is contained in the hyperplane $[{\Bbb E}]$, where the intersection is in the sense of counting multiplicity, that is, if $p_{1} = p_{2}$, then
\[ \phi(p_{1}) \cap \phi(p_{2}) = \{ \sigma \in H^{0}(X, K_{X} \otimes L) \mid \text{$p_{1}$ is a zero of $\sigma$ with multiplicity $2$} \}\]
Note that $\phi(p_{1}) \cap \phi(p_{2}) \cap \cdots \cap \phi(p_{k})$ is the image of the canonical homomorphism
\[ H^{0}\big(X, K_{X} \otimes L \otimes \mathcal{O}_{X}(-p_{1} - \cdots -p_{k}) \big) \rightarrow H^{0}\big(X, K_{X}\otimes L \big). \]
By using Serre duality, we know that ${\Bbb E}$ is contained in the kernel of the natural dual homomorphism
\[ H^{1}\big(L^{-1} \otimes \mathcal{O}_{X}(p_{1} + \cdots + p_{k})\big) \leftarrow H^{1}(X, L^{-1}), \]
and vice versa.
\end{proof}

\begin{rem}
\label{rem:sec1}
If $L$ is a line bundle of degree $2$, then $\phi$ is not necessarily an embedding. However {\rm Sec}$_{1}(X)$ is still well defined and Lemma \ref{lem:secant} is also valid for $k=1$.
\end{rem}

\begin{proof}[Proof of Theorem \ref{thm:stable}]
Let $L = L_{2} \otimes L_{1}^{-1}$ and $d = \deg L = d_{2} - d_{1} > 0$. It suffices to show that there exists a stable extension of $L$ by $\mathcal O_{X}$. By the Riemann-Roch theorem, we have
\[ \dim \operatorname{Ext}_{X}^{1} (L, \mathcal O_{X}) = \dim H^{1}(X, L^{-1}) = g_{X} - 1 + \deg L = g_{X} - 1 + d. \]
Suppose that $E$ is not a stable extension of $L$ by $\mathcal O_{X}$.
Then there exists a line subbundle $F$ of $E$ such that $\deg F \geq \lceil \frac{d}{2} \rceil > 0$, where $\lceil \frac{d}{2} \rceil$ be the minimal integer $\geq \frac{d}{2}$. Since any map from $F$ to $\mathcal O_{X}$ must be a zero map, looking at the following commutative diagram
\[\xymatrix{
& & F\ar[d]^{j} \ar[dr]^{\psi} & & \\
0\ar[r] & \mathcal O_{X} \ar[r]^{i} & E \ar[r] ^{p} & L \ar[r] & 0,
}\]
we can see that $\psi = p \circ j$ must be nonzero, which implies $\lceil \frac{d}{2} \rceil \leq \deg F \leq \deg L = d$.

Suppose that $\deg F = d - \ell $ with $0 \leq \ell \leq d - \lceil \frac{d}{2} \rceil $. Then $F \cong L \otimes \mathcal O_{X}(-p_{1} - \cdots - p_{\ell})$ for some points $p_{1}, \cdots, p_{\ell} \in X$. Consider the following commutative diagram of sheaves
\[ \xymatrix{
0 \ar[r] & \mathcal{H}om(L, \mathcal O_{X}) \ar@{->}[d] \ar[r] & \mathcal{H}om(L, E)\ar[d]\ar@{->}[d]  \ar[r] &\mathcal{H}om(L, L) \ar[d]\ar@{->}[d] \ar[r] &0 \\
0 \ar[r] & \mathcal{H}om(F, \mathcal O_{X}) \ar[r]                         & \mathcal{H}om(F, E) \ar[r]                                 &\mathcal{H}om(F, L) \ar[r]                                  &0,
} \]
where all the vertical homomorphisms are induced by $\psi$. Taking the long exact sequence of the cohomology with respect to the above short exact sequence of the sheaves, we have
\[ \xymatrix{
0 \ar[r] & Hom(L, \mathcal O_{X}) \ar@{->}[d] \ar[r] & Hom(L, E)\ar[d]\ar@{->}[d]  \ar[r] &Hom(L, L) \ar[d]^{\psi_{0}}\ar@{->}[d] \ar[r]^{\delta_{L}} & \operatorname{Ext}^{1}_{X}(L, \mathcal O_{X})\ar[d]^{\psi_{1}} \\
0 \ar[r] & Hom(F, \mathcal O_{X}) \ar[r]                         & Hom(F, E) \ar[r]                                 &Hom(F, L) \ar[r]^{\delta_{F}}                                  &\operatorname{Ext}^{1}_{X}(F, \mathcal O_{X}).
} \]
Then $\psi_{1} \circ \delta_{L} = \delta_{F} \circ \psi_{0}$ and $\psi = \psi_{0} ({\rm id}_{L})$. Note that $\delta_{F}(\psi) = 0$ since $\psi = p \circ j$. Hence $\psi_{1} ({\Bbb E}) = 0$, where ${\Bbb E} = \delta_{L} ({\rm id}_{L})$ is the extension of $L$ by $\mathcal O_{X}$, i.e. ${\Bbb E} \in \ker (\psi_{1})$. Since
\[ \operatorname{Ext}_{X}^{1}(L, \mathcal O_{X}) \cong H^{1}(X, L^{-1}), \quad \operatorname{Ext}_{X}^{1}(F, \mathcal O_{X}) \cong H^{1}(X, L^{-1} \otimes \mathcal O_{X}(p_{1} + \cdots + p_{\ell})), \]
$\mathbb{E}$ lies in the kernel of the canonical map
\[ H^{1}(X, L^{-1}) \to H^{1}(X, L^{-1} \otimes \mathcal O_{X}(p_{1} + \cdots + p_{\ell})), \]
i.e. $\mathbb{E}$ is contained in ${\rm Sec}_{\ell}(X)$ by Lemma \ref{lem:secant}. Moreover, the above map is surjective since
\[ H^{1} \left(X, \oplus_{i=1}^{\ell} \mathbb C_{p_{i}} \right) = 0.\]

\begin{itemize}
\item If $d =1$, then $\ell = 0$ and $\deg F = d = \deg L$. Hence, if $E$ is not a stable extension of $L$ by $\mathcal O_{X}$, then the extension is a holomorphic splitting, i.e. $E \cong \mathcal O_{X} \oplus L$. On the other hand, $\dim \operatorname{Ext}_{X}^{1} (L, \mathcal O_{X}) = g_{X} \geq 1$, which means each extension in $\operatorname{Ext}_{X}^{1} (L, \mathcal O_{X}) \setminus \{ 0 \}$ is stable.
\item If $d = 2$, then by Remark \ref{rem:sec1} and $\dim \mathcal R = d + g_{X} - 2 \geq 2 > \dim \operatorname{Sec}_{1}(X)$. We are done.
\item If $d \geq 3$, then by Lemma \ref{lem:secant} and the above argument, we know that all the unstable extensions of $L$ by $\mathcal O_{X}$ lie in the variety Sec$_{\lfloor \frac{d}{2} \rfloor}(X)$ in $\mathcal R$, where $\lfloor \frac{d}{2} \rfloor = d - \lceil \frac{d}{2} \rceil$. Note that
\[ \dim {\rm Sec}_{\lfloor \frac{d}{2} \rfloor}(X) \leq 2 {\left \lfloor \frac{d}{2} \right\rfloor} - 1 = \left\{ \begin{array}{ll}
d - 2 & \text{if $d$ is odd}, \\
d - 1 & \text{if $d$ is even}.
\end{array} \right. \]
Hence
\[ \dim \mathcal R - \dim {\rm Sec}_{\lfloor \frac{d}{2} \rfloor}(X) \geq \left\{ \begin{array}{ll}
g_{X} & \text{if $d$ is odd}, \\
g_{X} - 1 & \text{if $d$ is even}.
\end{array} \right. \]
Therefore, we complete the proof of Theorem \ref{thm:stable}. \qedhere
\end{itemize}
\end{proof}

\section{Ramification divisor maps}
\label{sec:ramidivmap}

Consider the following stable extension of two line bundles
\[ {\Bbb E}: \quad 0 \to L \xrightarrow{i} E \xrightarrow{p} M \to 0 \]
on a compact Riemann surface $X$ of genus $g_{X} > 0$ with a K\" ahler form $\omega_X$.
Since $E$ is stable, there exists on $E$ a unique Hermitian-Einstein metric $h$ up to scaling such that the corresponding Chern connection $D_{E}$ satisfies
\[ \Theta_{E} = D_{E} \circ D_{E} =  \lambda I, \quad\text{where}\quad \lambda= -\sqrt{-1} \,(\deg E) \,\omega_X. \]
As before, let $\bar{\partial}_{E}$ be the complex structure of $E$, and denote by $\mathcal{A}^{p,q}(E)$ the sheaf of smooth $E$-valued $(p, q)$-forms. Then we denote
\[ \partial_{E} := \left(D_{E} - \bar{\partial}_{E}\right) \colon \mathcal A^{0}(E) \to \mathcal A^{1,0}(E). \]
Since both $L$ and $M$ inherit hermitian structures from $(E,\, h)$, we could write the Chern connection $D_E$ on $(E,h)$ as
\[ D_{E} = \begin{pmatrix}
D_{L} &  - \beta \\
  \beta^{*_{\beta}}     &D_{M}
\end{pmatrix}, \]
where $\beta$ is a smooth $\operatorname{Hom}(M, L)$-valued $(0,1)$-form, and $\beta^{*_{\beta}}$ is the adjoint form of $\beta$ with respect to the hermitian metrics on both $L$ and $M$. Hence $\beta^{*_\beta}$ turns out to be a smooth $\operatorname{Hom}(L, M)$-valued $(1,0)$-form, which is called  the {\it second fundamental form} of $L$ in $E$.  Therefore, we could write $\partial_E$ as
\[ \partial_{E}= \begin{pmatrix}
\partial_{L} & 0 \\
  \beta^{*_{\beta}}     &\partial_{M}
\end{pmatrix}. \]
Recall the following commutative diagram in Subsection \ref{subsec:rami}
\[\xymatrix{
\mathcal A^{0}(L)\ar[r] ^-{i}\ar[drr]_-{\theta_{\Bbb E}} & \mathcal A^{0}(E)\ar[r] ^-{\partial_{E}} & E\otimes \mathcal{A}^{1,0}\ar[d]^-{p \otimes \operatorname{id}} \\
 & & M \otimes \mathcal{A}^{1,0}
}\]
where $p \otimes \operatorname{id}$ is induced by $E \rightarrow E/L \cong M$ and $\theta_{\Bbb E} = (p \otimes \operatorname{id}) \circ \partial_{E} \circ i$.
Then for any smooth section $s \colon X \to L$, we have $ \partial_{E}(i(s)) = i(\partial_{L}(s)) + \beta^{*_{\beta}}(s)$, which implies $(p\circ \operatorname{id})(\partial_{E}(i(s))) = \beta^{*_{\beta}}(s)$, that is, $\theta_{\mathbb E} = \beta^{*_{\beta}}$.

In order to prove Theorem \ref{thm:rami}, we need the following two lemmas.

\begin{lem}
\label{lem:flatcase}
Let $E$ be a rank two stable vector bundle on $X$ of genus $g_{X} \geq 1$ with $\det E = \mathcal O_{X}$, and let $L$ be a line subbundle of $E$. If $D$ is the effective divisor represented by the irreducible metric corresponding to the embedding $L \to E$, then $\theta_{\Bbb E}$ is ${\mathcal O}_X$-linear and its associated divisor $\operatorname{Div}(\theta_{\mathbb E})$ coincides with $D$.
\end{lem}
\begin{proof}
By the stability and flatness of $E$, there exists an open covering $\{ U_{\alpha} \}$ of $X$ and constant transition functions
\[ g_{\alpha \beta} = \begin{pmatrix}
a_{\alpha \beta} & b_{\alpha \beta} \\
c_{\alpha \beta} & d_{\alpha \beta}
\end{pmatrix} \in {\rm SU}(2) \subset {\rm SL}(2, \mathbb C) \]
on $U_{\alpha\beta} = U_{\alpha} \cap U_{\beta}$. Let $(e_{1, \alpha}, e_{2, \alpha})$ be a holomorphic flat frame of $E|_{U_{\alpha}}$ relative to $\{ U_{\alpha \beta}, g_{\alpha \beta}\}$. Then the Chern connection $D_{E}$ is given by
\[ D_{E} \colon \sum_{i=1,2} f_{i, \alpha} e_{i, \alpha} \mapsto \sum_{i=1,2} d f_{i, \alpha} \otimes e_{i, \alpha} \]
on $\mathcal{A}^{0}(E)|_{U_{\alpha}}$ and
\[ \partial_{E} \left(\sum_{i=1,2} f_{i, \alpha} e_{i, \alpha} \right) = \sum_{i=1,2} \partial f_{i, \alpha} \otimes e_{i, \alpha}. \]
Consider the following commutative diagram
\begin{equation}
\label{equ:initialcase}
\begin{split} \xymatrix{
 L\ar[r] ^-{i}\ar[drr]_-{\theta_{\mathbb E}} & E\ar[r] ^-{\partial_{E}} & E\otimes K_{X}\ar[d]^-{p \otimes \operatorname{id}} \\
           & & L^{-1}\otimes K_{X},
} \end{split}\end{equation}
where $p \otimes \operatorname{id}$ is induced by $E \rightarrow E/L \cong L^{-1}$, and $\theta_{\mathbb E} = (p\otimes \operatorname{id}) \circ \partial_{E} \circ i$. Let $\{ U_{\alpha}, w_{\alpha} \}$ be the local expressions of the section $s \colon X \to \mathbb P(E)$ corresponding to $L \to E$. Without loss of generality, we could assume that $w_{\alpha}$ is a holomorphic map from $U_{\alpha}$ to $\mathbb{C}$. Then $w_{\alpha} = \frac{a_{\alpha \beta}w_{\beta} + b_{\alpha \beta}}{c_{\alpha \beta}w_{\beta} + d_{\alpha \beta}}$ and $\sigma_{\alpha} = w_{\alpha}e_{1, \alpha} + e_{2, \alpha}$ is a frame of $L|_{U_{\alpha}}$. Moreover, the data $\{ U_{\alpha}, w_{\alpha} \}$ define a branched projective covering on $X$. It can be checked directly that $L$ is defined by transition functions $l_{\alpha \beta} = c_{\alpha \beta} w_{\beta} + d_{\alpha \beta}$. By Lemma \ref{lem:nonflat}, $s$ is not locally flat, i.e. $w_{\alpha}$ is not a constant map. Let $\{ U_{\alpha}, f_{\alpha}\sigma_{\alpha} \}$ be a meromorphic section of $L$ with respect to the transition functions $l_{\alpha \beta}$. Then $i (f_{\alpha} \sigma_{\alpha}) = f_{\alpha}w_{\alpha}\ e_{1, \alpha} + f_{\alpha} \ e_{2, \alpha}$ and
\begin{align*}
\partial_{E} (f_{\alpha}w_{\alpha} \, e_{1, \alpha} + f_{\alpha} \, e_{2, \alpha}) &= df_{\alpha} \otimes w_{\alpha} \, e_{1, \alpha} + f_{\alpha} \, dw_{\alpha} \otimes e_{1, \alpha} + df_{\alpha} \otimes e_{2, \alpha} \\
 &= df_{\alpha} \otimes \sigma_{\alpha} + f_{\alpha} \, dw_{\alpha} \otimes e_{1, \alpha}, \\
p \otimes \operatorname{id}\, (df_{\alpha} \otimes \sigma_{\alpha} + f_{\alpha} \, dw_{\alpha} \otimes e_{1, \alpha}) &= f_{\alpha} \, dw_{\alpha} \otimes \bar{e}_{1, \alpha} \\
&=f_{\alpha} \, w'_{\alpha} \otimes \bar{e}_{1, \alpha} dz_{\alpha},
\end{align*}
where $\bar{e}_{1, \alpha}dz_{\alpha}$ is a frame of $(L^{-1}\otimes K_{X})|_{U_{\alpha}}$. Hence
\[ \theta_{\mathbb E}(\sigma_{\alpha}) = w'_{\alpha} \otimes \bar{e}_{1, \alpha}dz_{\alpha},\quad \theta_{\mathbb E}(f_{\alpha}\sigma_{\alpha}) =f_{\alpha} \theta_{\mathbb E}(\sigma_{\alpha}).  \]
They imply that $\theta_{\mathbb E} \colon L \rightarrow L^{-1}\otimes K_{X}$ is $\mathcal{O}_{X}$-linear. (We have verified the first statement of Theorem \ref{thm:rami} in this special case). Therefore $\theta_{\mathbb E}$ can be viewed as a holomorphic section $\{ U_{\alpha}, w'_{\alpha}dz_{\alpha} \}$ of $\mathcal{H}om(L, L^{-1}\otimes K_{X}) = L^{-2} \otimes K_{X}$ with
\[ w'_{\alpha}dz_{\alpha} = l_{\alpha \beta}^{-2} \; w'_{\beta}dz_{\beta} \]
on $U_{\alpha \beta}$. Hence the ramification divisor $B_{\{U_{\alpha, w_{\alpha}}\}} = \operatorname{Div}(\theta_{\mathbb E})$ by Remark \ref{rem:divisorequ}.
\end{proof}

\begin{rem} {\rm
\label{rem:polystablecase}
Suppose $E$ is a polystable vector bundle of rank two with $\det E = \mathcal O_{X}$, and $L$ is a line subbundle of $E$ such that the section $s_{(L, E)} = \{ U_{\alpha}, w_{\alpha} \} \colon X \to \mathbb P(E)$ is {\it not} locally flat. By a similar argument, the ramification divisor $B_{\{U_{\alpha, w_{\alpha}}\}} = \operatorname{Div}(\theta_{\mathbb E})$.
}\end{rem}

\begin{lem}
\label{lem:doublecover}
Let $X$ be a compact Riemann surface of genus $g_{X} \geq 1$. Then there exists a compact connected Riemann surface $\widehat{X}$ of genus $\left( 2g_{X} - 1 \right)$ with an \'etale double cover $\pi \colon \widehat{X} \rightarrow X$.
\end{lem}
\begin{proof}
Since $g_X \geq 1$, we could choose a non-trivial line bundle $L$ such that $L\otimes L \cong {\mathcal O}_X$. Defining
\[ \widehat{X} = \Big\{ \big(x, \tau(x)\big)\in L \mid x \in X, \tau(x) \in L|_{\{x\}} \text{ such that } \tau(x)^{2} = 1\Big\}, \]
as Exercise 1 in \cite[Chapter 2]{Voisin02}, we know that the natural projection $\pi\colon \widehat{X} \rightarrow X$ is an unramified double cover. Suppose that $\widehat{X}$ is not connected. Then $\widehat{X}$ must be a disjoint union of two copies of $X$, i.e.
$\widehat{X} = X_{1} \sqcup X_{2}$,
where $X_{1} \cong X \cong X_{2}$. The isomorphism $X \rightarrow X_{1} \subset \widehat{X}$ gives a nowhere vanishing section of $L$, which contradicts that $L$ is non-trivial.
\end{proof}

By using  Lemmas \ref{lem:flatcase} and \ref{lem:doublecover}, we could complete the proof of Theorem \ref{thm:rami}.
\begin{proof}[Proof of Theorem \ref{thm:rami}]
Note that
\begin{equation*}
\label{equ:curvature}
\begin{split}
\lambda I = D_{E} \circ D_{E} &= \begin{pmatrix}
D_{L} \circ D_{L} - \beta \wedge \beta^{*_{\beta}} & - \left(D_{L} \circ \beta + \beta \circ D_{M}\right) \\
\beta^{*_{\beta}} \circ D_{L} + D_{M} \circ \beta^{*_{\beta}} & -\beta^{*_{\beta}} \wedge \beta + D_{M} \circ D_{M}
\end{pmatrix} \\
&= \begin{pmatrix}
\Theta_{L} - \beta \wedge \beta^{*_{\beta}} & -D_{\mathcal{H}om(M, L)} (\beta)\\
D_{\mathcal{H}om(L, M)}(\beta^{*_{\beta}}) & -\beta^{*_{\beta}} \wedge \beta + \Theta_{M}
\end{pmatrix}.
\end{split}
\end{equation*}
Therefore $D_{L^{-1} \otimes M}(\beta^{*_{\beta}}) = 0$ and then $\bar{\partial}_{L^{-1} \otimes M}(\beta^{*_{\beta}}) = 0$. Hence $\beta^{*_{\beta}}$ is holomorphic which implies that $\theta_{\mathbb E} = \beta^{*_{\beta}} \in \operatorname{Hom}_{\mathcal O_{X}}(L, M \otimes K_{X})$.

Now, we prove the second assertion in Theorem \ref{thm:rami}. At first, we consider a polystable extension ${\Bbb E}:\, 0\to L\to E\to M \to 0$, where $E$ is a rank two polystable vector bundle of even degree. Then there is a line bundle $\eta$ such that $E \otimes \eta^{-1} := E_{0}$ is flat and $\det E_{0} = \mathcal O_{X}$. Moreover, $E/L \cong L^{-1} \otimes \eta^{2}$. Denote by $D_{E}$ and $D_{E_{0}}$ the Chern connections of $E$ and $E_{0}$ respectively. Then
\begin{align*}
\bar{\partial}_{E} & = \bar{\partial}_{E_{0}} \otimes \operatorname{id} + \operatorname{id} \otimes \bar{\partial}_{\eta}, \\
D_{E} & = D_{E_{0}} \otimes \operatorname{id} + \operatorname{id} \otimes D_{\eta}.
\end{align*}
Hence
\[ \partial_{E} = \partial_{E_{0}} \otimes \operatorname{id} + \operatorname{id} \otimes \partial_{\eta}. \]
The commutative diagram
\[\xymatrix{
 \mathcal A^{0}(L)\ar[r] ^-{i}\ar[drr]_-{\theta_{\mathbb E}} & \mathcal A^{0}(E)\ar[r] ^-{\partial_{E}} & E\otimes \mathcal{A}^{1,0}\ar[d]^-{p \otimes \operatorname{id}} \\
  & & L^{-1}\otimes \eta^{2} \otimes \mathcal{A}^{1,0},
}\]
can be rewritten as
\[\xymatrix{
 \mathcal A^{0}(L_{0} \otimes \eta)\ar[r] ^-{i}\ar[drr]_-{\theta_{\mathbb E}} & \mathcal A^{0}(E_{0}\otimes \eta)\ar[r] ^{\partial_{E_{0}} \otimes \operatorname{id} + \operatorname{id} \otimes \partial_{\eta}} & E_{0} \otimes\eta\otimes \mathcal{A}^{1,0}\ar[d]^-{p \otimes \operatorname{id}} \\
  & & {}\phantom{changying}L_{0}^{-1}\otimes \eta \otimes \mathcal{A}^{1,0},\phantom{changying}
}\]
where $L_{0} = L\otimes \eta^{-1}$ is a line subbundle of $E_{0}$. Denote by ${\Bbb E}_0$ the stable extension $0\to L_0\to E_0\to L_0^{-1}\to 0$.
Then $\theta_{\mathbb E} = \theta_{\mathbb E_{0}} \otimes \operatorname{id}_{\eta}$, which means that $\theta_{\mathbb E}$ and $\theta_{\mathbb E_{0}}$ are consistent as a section of $L_{0}^{-2} \otimes K_{X}$. Then, by  Lemma \ref{lem:flatcase}, we know that $D = \operatorname{Div}(\theta_{\mathbb E_{0}}) = \operatorname{Div}(\theta_{\mathbb E})$ since the two embeddings $L \to E$ and $L_{0} \to E_{0}$ are equivalent.


Suppose $E$ is stable and $\deg E$ is odd. Then we consider the \'etale double cover $\pi \colon \widehat{X} \to X$ given by Lemma \ref{lem:doublecover}. Let $\{ U_{\alpha}, w_{\alpha} \}$ be the local expression of $s_{(L,E)}$, which defines a branched projective covering of $X$.  Without loss of generality, we may assume that $U_{\alpha}$ is sufficiently small such that $\left\{\pi^{-1}(U_{\alpha}), \pi^*(w_{\alpha})\right\}$ is a branched projective covering of $\widehat{X}$. Then $\pi^*(\mathbb P(E)) = \mathbb{P}(\pi^*(E))$ is the indigenous bundle associated to $\{\pi^{-1}(U_{\alpha}), \pi^*(w_{\alpha})\}$ and
\[ \pi^{*} B_{\{U_{\alpha}, w_{\alpha}\}} = B_{\{\pi^{-1}(U_{\alpha}), \pi^*(w_{\alpha})\}}. \]
By the following diagram
\[\xymatrix{
 \mathcal A^{0}(\pi^{*}(L))\ar[r] \ar[drr]_-{\theta_{\pi^{*}({\mathbb E})}} & \mathcal A^{0}(\pi^{*}(E))\ar[r] ^{\partial_{\pi^{*}(E)}} & \pi^{*}(E) \otimes \mathcal{A}^{1,0}\ar[d] \\
 & & \pi^{*}(M) \otimes \mathcal{A}^{1,0},
}\]
where $M = E / L$, we obtain $\operatorname{Div}(\theta_{\pi^{*}({\mathbb E})}) = \pi^{*} \operatorname{Div}(\theta_{\mathbb E})$. Let $g$ be the irreducible metric defined by the stable embedding $L\to E$. We observe that $\pi^*g$ is a cone spherical metric representing the effective divisor $\pi^*D$ of even degree on $\widehat X$. Since the metric $\pi^*g$ is given by the embedding $\pi^*L\to \pi^*E$ and may be either reducible or irreducible, by Theorem \ref{thm:correspondence}, we have that $\pi^{*}(E)$ is  polystable and
\[ \operatorname{Div}(\theta_{\pi^{*}({\mathbb E})}) = B_{\{\pi^{-1}(U_{\alpha}), \pi^*(w_{\alpha})\}}. \]
Hence $\pi^{*} \operatorname{Div}(\theta_{\mathbb E}) = \pi^{*} B_{\{U_{\alpha}, w_{\alpha}\}}$. Therefore we obtain $\operatorname{Div}(\theta_{\mathbb E}) = B_{\{U_{\alpha}, w_{\alpha}\}} = D$ since $\pi$ is \'etale.

In summary, we obtain that if $E$ is a rank two stable vector bundle and $L$ is a line subbundle of $E$, then the irreducible metric defined by the embedding $L \to E$ represents the divisor $\operatorname{Div}(\theta_{\mathbb E})$.
\end{proof}

\begin{proof}[Proof of Theorem \ref{thm:expressionofR}]
Let $\mathbb{E}$ be a stable extension $0 \to L \xrightarrow{i} E \xrightarrow{p}  M \to 0$. Then the Chern connection $D_{E}$ on $E$ has the form of
\[ D_{E} = \begin{pmatrix}
D_{L} &  - \beta \\
  \beta^{*_{\beta}}     &D_{M}
\end{pmatrix}. \]
By chasing the diagram and using the Dolbeault isomorphism, 
we could obtain that $\beta$ is a representative 1-cocycle of the extension ${\Bbb E}$ in $H^{0,1}_{\bar{\partial}}(X, L \otimes M^{-1}) \cong \operatorname{Ext}_{X}^{1}(M, L)$. Let $H_{\bar{\partial}}^{0,1}(X, L \otimes M^{-1})^{s}$ be the Zariski open subset consisting of stable extensions in $H^{0,1}_{\bar{\partial}}(X, L \otimes M^{-1})$. Then we could define a map
\begin{align*}
\Phi \colon H_{\bar{\partial}}^{0,1}(X, L \otimes M^{-1})^{s} &\to H_{\bar{\partial}}^{1,0}(X, L^{-1} \otimes M) \\
  \beta &\mapsto \beta^{*_{\beta}}.
\end{align*}
Note that if $\mathbb{E}$ is the extension $0 \to L \xrightarrow{i} E \xrightarrow{p} M \to 0$ corresponding to $\beta$, then $\mu \mathbb{E}$ is the extension $0 \to L \xrightarrow{i} E \xrightarrow{p/\mu} M \to 0$ corresponding to $\mu \beta \; (\mu \in \mathbb{C}^{*})$. Hence,
\[ \Phi(\mu \beta) = \theta_{(\mu \mathbb E)} = (p/\mu \otimes \operatorname{id}) \circ \partial_{E} \circ i = \frac{\theta_{\mathbb E}}{\mu}  = \frac{\beta^{*_{\beta}}}{\mu}. \]
Then $\Phi$ naturally induces the ramification map
\begin{align*}
{\frak R}_{(L, M)}\colon \mathbb{P}\Big(H^{1}\big(X, L \otimes M^{-1}\big)^{s}\Big) &\rightarrow \mathbb{P}\Big(H^{0}\big(X, L^{-1} \otimes M \otimes K_{X}\big)\Big), \\
[{\Bbb E}] &\mapsto {\rm Div}(\theta_{\Bbb E}).
\end{align*}
modulo the two  $\mathbb C^{*}$-actions on the domain and the target of $\Phi$, respectively.

The real analyticity of the map
${\frak R}_{(L,\, M)}\colon [\beta]\mapsto [\beta^{*_\beta}]$ follows from that the Kobayashi-Hitchin correspondence is real analytic and both $\beta$ and $\beta^{*\beta}$ are harmonic.
\end{proof}

\begin{proof}[Proof of Corollary \ref{cor:image}]
The first statement of the corollary follows automatically from the definition of the ramification divisor map
${\frak R}_{(L,M)}$. By Theorems \ref{thm:stable} and \ref{thm:expressionofR},
${\frak R}_{(L,M)}$ is real analytic and defined on a Zariski open subset of $\mathbb{P}\Big(H^{1}\big(X, L \otimes M^{-1}\big)\Big)$ which is arcwise connected. Then ${\rm Im}\big({\frak R}_{(L,M)}\big)$ is an arcwise connected Borel subset in $\mathbb{P}\Big(H^{0}\big(X, L^{-1} \otimes M \otimes K_{X}\big)\Big)$ and we could talk about its Hausdorff dimension. The second statement of the corollary is reduced to the following claim.
\end{proof}

\noindent {\bf Claim}: {\it As $k=(\deg\, E-2\deg\, L) > 10g_{X} - 5$, there exists a stable extension ${\Bbb E}_0$ at which the tangent map ${\rm d}{\frak R}_{(L,\,M)}$ has rank $\geq 2\big(k+1-2g_X\big)$.}
\begin{proof} We use the notations in the proof of Theorem \ref{thm:expressionofR} to prove the claim. Without loss of generality, we may assume $\det E$ is isomorphic to a fixed line bundle $\xi$ of degree $0$ or $1$. By the proof of Theorem \ref{thm:corr}, as $k > 10g_{X} - 5$, the forgetful map ${\mathcal F}$ from $H^{1}\big(X, L \otimes M^{-1}\big)^{s}$ to the moduli space ${\mathcal M}_X(2,\,\xi)$ of rank two stable bundles with determinant $\xi$ on $X$ is surjective, which gives a fibration over ${\mathcal M}_X(2,\,\xi)$. Moreover, for a generic stable bundle $E_0$ in ${\mathcal M}_X(2,\,\xi)$, the fiber ${\mathcal F}^{-1}(E_0)$ is a Zariski open subset of $H^0(X,\,L^{-1}\otimes E_0)$ with complex dimension $(k+2-2g_X)$.
Let $h_0$ be the unique Hermitian-Einstein metric up to scaling on $E_0$ with respect to the K\"ahler form $\omega_X$ on $X$. Take a stable extension ${\Bbb E}_0\in {\mathcal F}^{-1}(E_0)$  defined by the harmonic form $\beta_0$ in ${\mathcal H}_{\bar \partial}^{0,1}(X,\,L\otimes M^{-1})$ with respect to the metric $h_0$. Moreover, by the definition of ${\mathcal F}^{-1}(E_0)$, $\beta_0$ could also be thought of as a nowhere vanishing section of $L^{-1}\otimes E$. Take a tangent vector $\alpha$ in the tangent space ${\rm T}_{{\Bbb E}_0}{\mathcal F}^{-1}(E_0)$  of ${\mathcal F}^{-1}(E_0)$ at ${\Bbb E}_0$. We could look at ${\rm T}_{{\Bbb E}_0}{\mathcal F}^{-1}(E_0)$ as $H^0(X,\,L^{-1}\otimes E_0)$, which is a complex linear subspace of $H^0(X,\,L^{-1}\otimes M)\cong {\mathcal H}_{\bar \partial}^{0,1}(X,\,L\otimes M^{-1})$.  As the complex parameter $t$ has sufficiently small norm, $\beta_0+t\alpha$ is also a nowhere section of $L^{-1}\otimes E$ and gives a family of stable extensions lying in ${\mathcal F}^{-1}(E_0)$. By Theorem \ref{thm:expressionofR}, restricting the map $\Phi$ to ${\mathcal F}^{-1}(E_0)$, we have that $$\Phi(\beta_0+t\alpha)=(\beta_0+t\alpha)^{*_{h_0}}=\beta_0^{*_{h_0}}+\bar t \alpha^{*_{h_0}}\quad {\rm as}\quad |t| \ll 1.$$ Hence the restriction of the tangent map ${\rm d}\Phi$ to ${\rm T}_{{\Bbb E}_0}{\mathcal F}^{-1}(E_0)$ is a non-degenerate complex anti-linear map. Then we complete the proof by the definition of ${\frak R}_{(L,\,M)}$ from $\Phi$.
\end{proof}

\section{Further discussion}

We may compare Corollary \ref{cor:finitemetric} with a theorem of A. Eremenko \cite{Er1905}, which says that for each $\lambda\in {\Bbb C} \setminus \{0,1\}$, there exist at most finitely many cone spherical metrics representing
the ${\Bbb R}$-divisor $\beta_1[0]+\beta_2[1]+\beta_3[\lambda]+\beta_4[\infty]$
on ${\Bbb P}^1$ such that each of $\beta_1,\cdots,\beta_4$  is a non-integer greater than $-1$.
Our proof of the corollary used the framework in Theorem \ref{thm:corr}, which is completely different from the function theoretical proof of Eremenko's theorem. Motivated by \cite[Theorem 1.5]{BdMM11} and \cite{CLWpre}, we speculate that the ``a.e.'' condition could be removed in Corollary \ref{cor:finitemetric} and there should exist uniform lower and upper bounds for the number of irreducible metrics representing a given effective divisor $D$ with odd degree in terms of $\deg\, D$ and $g_X$. This question has an interesting special case that the ramification divisor map ${\frak R}_{({\mathcal O}_X,\,M)}$ is a real analytic surjective map from ${\Bbb P}^{g_X-1}$ to itself, where $M\in {\rm Pic}^1(X)$.

We also compare Corollary \ref{cor:irr} with a theorem of Mondello-Panov \cite[Theorem D]{MP1807}. Their theorem says that for each compact Riemann surface $X$ of genus $g_{X}$ with $n$ marked points $p_1,\cdots, p_n$ in the Riemann moduli space ${\cal M}_{g_{X},n}$ with $2 - 2g_{X} - n < -1$, there exists a vector $\vartheta=(\vartheta_1,\cdots,\vartheta_n)$ of $n$ cone angles such that there is no spherical metric representing $D=\sum_{j=1}^n\,(\vartheta_j-1)p_j$ on $X$, but there does exist an irreducible metric with the given $n$ cone angles on another Riemann surface of genus $g_{X}$. Moreover, the vector $\vartheta$ of cone angles satisfies the non-bubbling condition ${\rm NB}_\vartheta (g_{X}, n) > 0$ \cite[Definition 1.5]{MP1807}.
Though the vectors of cone angles in Corollary \ref{cor:irr} do not satisfy such condition, we obtained the existence of irreducible metrics representing some divisors $D_0$ in each complete linear system $|D|$ such that ${\rm even}=\deg\, D > 2g_{X} - 2$ on every Riemann surface of genus $g_{X} \geq 2$. Moreover, we could specify such $D_0$'s in terms of the ramification divisor map.

Recall the ramification divisor map
\begin{align*}
{\frak R}_{(L, M)}\colon \mathbb{P}\Big(H^{1}\big(X, L \otimes M^{-1}\big)^{s}\Big) &\rightarrow |L^{-1} \otimes M \otimes K_{X}|, \quad
[{\Bbb E}]\mapsto {\rm Div}(\theta_{\Bbb E}).
\end{align*}
in Corollary \ref{cor:image}.
Xuwen Zhu and B.X. \cite[Section 5]{XuZhu19} proposed a question of finding irreducible metrics which have bounded 2-eigenfunctions for the associated Laplace-Beltrami operators. Motivated by a conversation with Xuwen Zhu, we conjecture that
the points of $\mathbb{P}\Big(H^{1}\big(X, L \otimes M^{-1}\big)^{s}\Big)$  at which the tangential map of ${\frak R}_{(L, M)}$  is {\it not} surjective give examples of such irreducible metrics by using the correspondence $\sigma$ in Theorem \ref{thm:corr}. On the other hand, we expect that the tangential map of ${\frak R}_{(L, M)}$ should be surjective almost everywhere in $\mathbb{P}\Big(H^{1}\big(X, L \otimes M^{-1}\big)^{s}\Big)$. If it were true, by Corollary \ref{cor:image} and the inverse function theorem,  we could obtain a much stronger existence result of irreducible metrics than Corollary \ref{cor:irr} that  the image of ${\frak R}_{(L, M)}$ contains at least an Euclidean open subset in the complete linear system $|L^{-1} \otimes M \otimes K_{X}|$.

We consider another question about the ramification divisor map ${\frak R}_{(L, M)}:\,[{\Bbb E}]\mapsto {\rm Div}(\theta_{\Bbb E})$. Since it is a real analytic map defined on a Zariski open subset of $\mathbb{P}\Big(H^{1}\big(X, L \otimes M^{-1}\big)\Big)$, we would like to investigate the possible limiting metric of the images ${\frak R}_{(L, M)}({\Bbb E}_n)$ of a sequence $\{{\Bbb E_n}\}$ of stable extensions  which converges to an  unstable and non-trivial extension ${\Bbb E}$ in $\mathbb{P}\Big(H^{1}\big(X, L \otimes M^{-1}\big)\Big)$. In particular, as $(\deg\,M - \deg\,L)$ is even, we could further assume that ${\Bbb E}$ is unstable and polystable. We speculate that in the latter case we may obtain the convergence of irreducible metrics of even degree to reducible ones.

We observe that on a compact Riemann surface $X$ of genus $g_X>0$, there exists  a canonical Hermitian metric $H$  on  $H^{1}\big(X, L \otimes M^{-1}\big)^{s}$ defined as follows: Given two tangent vectors $\beta_1,\,\beta_2$ at an stable extension ${\Bbb E}={\Bbb E}_\beta$\,:\,$0\to L\to E\to M\to 0$, identify them with their harmonic representatives in $H^{0,1}_{\bar{\partial}}(X, L \otimes M^{-1})$ respectively, which we denote by the same notions,  and define $H(\beta_1,\,\beta_2)$ to be $\int_X\,\langle \beta_1,\,\beta_2 \rangle_{h_\beta} \omega_X$, where
$h_\beta$ is the pointwise Hermitian inner product on the space of $\big(L\otimes M^{-1}\big)$-valued $(0,1)$-forms induced by both the stable extension ${\Bbb E}_\beta$ and the Hermitian-Einstein metric $h$ on $E$. $H$ induces a canonical closed 2-form
$\omega_{(L,\,M)}:=\frac{\sqrt{-1}}{2\pi}\partial {\overline \partial}\big(\log\, H(\beta,\,\beta)\big)$ on the domain ${\rm Dom}\big({\frak R}_{(L,M)}\big)=\mathbb{P}\Big(H^{1}\big(X, L \otimes M^{-1}\big)^{s}\Big)$ of the ramification divisor map ${\frak R}_{(L,M)}$.
Is $\omega_{(L,\,M)}$ a {\it K\" ahler form} on ${\rm Dom}\big({\frak R}_{(L,M)}\big)$? This question is particularly tempting in the case that the genus $g_X > 1$ and $(L,\,M)$ lies in
${\mathcal O}_X\times {\rm Pic}^1(X)$, since the positive answer to it would give a new class $\left\{\omega_M\right\}$ of K\" ahler metrics on the projective space ${\Bbb P}^{g_X-1}$ by the proof of Corollary \ref{cor:finitemetric}.

We find
 that the moduli space $\mathcal{SE}(X)$ of stable extensions in Theorem \ref{thm:corr} forms a special example of the more general moduli of extensions of holomorphic bundles on K\"ahler manifolds in Daskalopoulos-Uhlenbeck-Wentworth \cite{DUW95}. Moreover, we also find \cite{Tian92} where
Tian established some stability property for the tangent bundles of Fano manifolds with K\"ahler-Einstein metrics, which contains a high dimensional generalization of stable extensions in Theorem \ref{thm:corr}.
We would like to borrow the ideas from both \cite{Tian92, DUW95} to attack the problems in the preceding paragraphs in the future.

At last, we observe that there exists a parallel correspondence between general irreducible metrics whose cone angles does {\it not} necessarily lie in $2\pi{\mathbb Z_{>0}}$ and parabolic line subbundles of rank two parabolic stable bundles with parabolic degree {\it zero}. We have been establishing an algebraic framework for general cone spherical metrics on a compact Riemann surface in an on-going work and will give the details in a future paper.

\begin{center}
 {\bf Acknowledgements}
\end{center}

\noindent B.X. would like to thank Gang Tian, Rafe Mazzeo, Song Sun, Xuwen Zhu and Andriy Haydys for many insightful discussions, and their hospitality during B.X.'s visit to Peking University in Winter 2018, Stanford University and UC Berkeley in Spring 2019 and to University of Freiburg in Summer 2019.
J.S. appreciates her hospitality during his visit in Summer 2021 to Institute of Geometry and Physics, USTC, where he and the other two authors finalized this manuscript.
All the authors would also like to thank Xuemiao Chen, Yifei Chen, Jiahao Hu, Qiongling Li,
Mao Sheng, Xiaotao Sun, Botong Wang, Zhenjian Wang, Chengjian Yao, Chenglong Yu, Lei Zhang, Mingshuo Zhou, and Kang Zuo for their helpful suggestions and comments. We appreciate  the anonymous referee
for her/his many valuable comments which improve the exposition greatly.

L.L. is supported in part by the National Natural Science Foundation of China (Grant No. 11501418).
J.S. is partially supported by National Natural Science Foundation of China (Grant Nos. 12001399 and 11831013).
B.X. is supported in part by the National Natural Science Foundation of China (Grant Nos. 11571330, 11971450 and 12071449).
Both J.S. and B.X. are supported in part by the Fundamental Research Funds for the Central Universities.

\vspace{0.5cm}

{\sc \noindent Lingguang Li\\
School of Mathematical Sciences\\
Tongji University\\
Shanghai 200092 China}\\
LiLg@tongji.edu.cn\\

{\small {\sc  \noindent Jijian Song\\
Center for Applied Mathematics\\
School of Mathematics, Tianjin University\\
Tianjin 300350 China}\\
\Envelope smath@tju.edu.cn\\

{\sc \noindent Bin Xu\\
CAS Wu Wen-Tsun Key Laboratory of Mathematics and\\
School of Mathematical Sciences\\
University of Science and Technology of China\\
Hefei 230026 China}\\
bxu@ustc.edu.cn}

\end{document}